\documentclass[11pt]{amsart}
\usepackage{amsmath,amssymb,latexsym,soul,cite,mathrsfs}

\usepackage{color,enumitem,graphicx}
\usepackage[colorlinks=true,urlcolor=blue,
citecolor=red,linkcolor=blue,linktocpage,pdfpagelabels,
bookmarksnumbered,bookmarksopen]{hyperref}
\usepackage[english]{babel}
\usepackage{comment}

\usepackage[left=2.9cm,right=2.9cm,top=2.8cm,bottom=2.8cm]{geometry}
\usepackage[hyperpageref]{backref}

\usepackage[colorinlistoftodos]{todonotes}
\makeatletter
\providecommand\@dotsep{5}
\def\listtodoname{List of Todos}
\def\listoftodos{\@starttoc{tdo}\listtodoname}
\makeatother

\usepackage{verbatim} 

\numberwithin{equation}{section}

\newtheorem{lem}{Lemma}
\newtheorem{prop}[lem]{Proposition}
\newtheorem{theo}{Theorem}
\newtheorem{coro}[lem]{Corollary}
\newtheorem{rem}[lem]{Remark}
\newcommand\restr[2]{{
  \left.\kern-\nulldelimiterspace 
  #1 
  \vphantom{\big|} 
  \right|_{#2} 
  }}

\newcommand\numberthis{\addtocounter{equation}{1}\tag{\theequation}}
\providecommand{\abs}[1]{\lvert#1\rvert}
\providecommand{\norm}[1]{\lVert#1\rVert}

\DeclareMathOperator{\supp}{supp}

\title[Least energy sign-changing solution for an asymptotically cubic SP system]
{Least energy radial sign-changing solution \\for the  Schr\"odinger-Poisson system in $\mathbb R^{3}$\\
under an asymptotically cubic nonlinearity
}

\author[E. G. Murcia]{ Edwin Gonzalo Murcia}
\author[G. Siciliano]{Gaetano Siciliano}

\address[E. G. Murcia]{\newline\indent
	Departamento de Matem\'aticas
	\newline\indent 
	Pontificia Universidad Javeriana
	\newline\indent
	Carrera 7 No. 43-82, Bogot\'a, Colombia}
\email{\href{mailto:murciae@javeriana.edu.co}{murciae@javeriana.edu.co}}

\address[G. Siciliano]{\newline\indent
	Departamento de Matem\'atica - Instituto de Matem\'atica e Estat\'istica
	\newline\indent 
	Universidade de S\~ao Paulo
	\newline\indent
	Rua do Mat\~ao 1010,  05508-090  S\~ao Paulo, Brazil}
\email{\href{mailto:sicilian@ime.usp.br}{sicilian@ime.usp.br}}

\thanks{Edwin Murcia  was supported by Department of Mathematics, Pontificia Universidad Javeriana.
Gaetano Siciliano was partially supported by Fapesp, CNPq and Capes, Brazil. 
}

\begin{document}

\begin{abstract}
In this paper we consider the following Schr\"odinger-Poisson system in the whole $\mathbb R^{3}$, 
\begin{equation*}
     \left\{
        \begin{array}{ll}
            -\Delta u+u+ \lambda \phi u=f(u) &\text{ in } \mathbb R^3, \\
            -\Delta \phi= u^2  &\text{ in } \mathbb R^3,
        \end{array}
    \right.                
\end{equation*}
where $\lambda>0$ and
 the nonlinearity $f$ is ``asymptotically cubic'' at infinity.
This implies that the nonlocal term $\phi u$ and the nonlinear term $f(u)$ are, in some sense, in a strict competition.
We show that the system admits a least energy sign-changing and radial solution obtained by minimizing the 
energy functional on the so-called {\sl nodal Nehari set}.

\end{abstract}

\maketitle

\section{Introduction}

A great attention has been given in the last decades to the so called Schr\"odinger-Poisson system, namely
\begin{equation}\label{ourproblem}
    \left\{
        \begin{array}{ll}
            -\Delta u+u+ \lambda \phi u=f(u) &\text{ in } \mathbb R^3, \smallskip \\
            -\Delta \phi= u^2  &\text{ in } \mathbb R^3,
        \end{array}
    \right.        
\end{equation}
due especially to its importance in many physical applications but also since it presents difficulties and
challenges from a mathematical point of view. 

It is known that the system can be reduced to the equation
$$-\Delta u+u+ \lambda \phi_{u} u=f(u) \text{ in } \ \ \mathbb R^3,$$
and that its solutions can be found as critical points in $H^{1}(\mathbb R^{3})$
of the energy functional
\begin{equation}\label{eq:I}
I(u) := \frac{1}{2}\int_{\mathbb R^{3}}\left(|\nabla u|^{2}+u^{2}\right)dx+\frac\lambda4 \int_{\mathbb R^{3}} \phi_{u}u^{2}dx
-\int_{\mathbb R^{3}}F(u) dx,
\end{equation}
where 
$$F(t) := \int_{0}^{t}f(\tau)d\tau, \quad \phi_{u} = \frac{1}{4\pi|\cdot|} *u^{2}.$$
Before anything else, we observe that  $\phi_{u}$ is automatically positive and univocally defined by $u$; hence
words like ``solution'', ``positive'', ``sign-changing''
 always refer to the unknown $u$ of the system.

Observe that
since $\phi_{u} u $ is $3-$homogeneous, in the sense that
$$ \phi _{t u} (tu) = t^{3} \phi_{u} u, \quad t\in \mathbb R,$$
there is a further difficulty in the problem exactly
when the nonlinearity $f$ behaves ``cubically'' at infinity,
we say it is {\sl asymptotically cubic}, being in this case in competition with the nonlocal term
$\phi_{u} u$.

The number of  papers which have studied the Schr\"odinger-Poisson system
in the mathematical
literature is so huge that it is almost impossible to give a satisfactory list. Indeed many papers deal with the problem in
bounded domain or  in the whole space (see e.g. \cite{BF,BS,DM,PS,RS,S} 
and the references therein) and some other papers deal with the fractional counterpart 
(see e.g \cite{MS,dS} and its references). In all the cited papers
various type of solutions have been found under different assumptions on the nonlinearity.
However the solutions found  are positive or with undefined sign
and the nonlinearity $f$ is ``supercubic'' at infinity
(in a sense that will be specified below) an this fact helps in many computations 
since it gains on the nonlocal term $\phi_{u}u$.

 Nevertheless some  results have been obtained
also in the asymptotically cubic case: for example, in the remarkable paper \cite{ADP} 
the authors consider the existence of solutions  under a very general nonlinearity
$f$ of Berestycki-Lions type. However they found a {\sl positive} solution,
for small values of the parameter $\lambda>0$ 
(and as we will see the smallness of $\lambda$ is necessary).

\medskip

However beside the existence of positive solutions it is also interesting to find sign-changing 
solutions and indeed many
 authors began recently to address this issue.
We cite  the interesting paper \cite{WZ} which deals with the case  $f(u)=|u|^{p-1}u$ and $p\in (3,5)$ and where
the authors search for least energy sign-changing solutions, that is the sign-changing solution whose functional 
has minimal energy among all the others sign-changing solutions.
 Their idea is to study the energy functional on a new
constraint, a subset of the Nehari manifold which contains all the sign-changing solutions. 

Another interesting paper is \cite{ASS} which deals with a more general nonlinearity, not necessarily of power type, 
where the authors assume that
\begin{itemize}
\item $\displaystyle\lim_{t\to\infty} \frac{F(t)}{t^{4}} =+\infty$.
\end{itemize}
In this sense  \cite{ASS} and \cite{WZ} deal with a {\sl supercubic} nonlinearity $f$.

The above condition is also required in  \cite{AS}, for the case of the bounded domain,
and in \cite{LXZ}, for the case of the whole space,
where a least energy sign-changing solution is obtained.
 
In all these papers concerning sign-changing solutions, one of the main task is to prove that the new constraint on which minimize the functional is not empty. To show this, the fact that the nonlinearity is supercubic
is strongly used.

\medskip

Motivated by the previous discussion, a natural question which arises concerns 
the case when the nonlinearity is ``cubic'' at infinity.
More specifically in this paper we address the problem \eqref{ourproblem}
under the following conditions. Let $\lambda>0$
and assume \smallskip
\begin{enumerate}[label=(f\arabic*),ref=f\arabic*,start=1]
\item\label{f_{1}} $f \in C (\mathbb R ,\mathbb R)$; \medskip
\item\label{f_{2}} $f(t)=-f(-t)$ for $t\in \mathbb R$; \medskip 
\item\label{f_{3}} $\lim_{t \rightarrow 0} {f(t)/t}=0$; \medskip
\item\label{f_{4}} 
    $\lim_{t\rightarrow \infty}{f(t)}/{t^3}=1$ and $f(t)/t^3<1$ for all $t\in \mathbb R$; \medskip
\item\label{f_{5}}  the function $ t\mapsto f(t)/t^3$ is strictly increasing on $(0,\infty)$; \medskip
\item\label{f_{6}} recalling that $F(t)=\int_0^t f(\tau)d \tau$,
    \[
        \lim_{t\rightarrow \infty} [f(t)t-4F(t)]=+\infty.
    \]
\end{enumerate}
Assumption \eqref{f_{4}} is what we called {\sl asymptotically cubic} behaviour for the nonlinearity and 
\eqref{f_{6}} is the analogous of the usual {\sl non-quadraticity condition}.

Our result is the following
\begin{theo}\label{theorem}
    If $\lambda>0$ is sufficiently small, under the conditions \eqref{f_{1}}-\eqref{f_{6}}, problem \eqref{ourproblem} has a radial sign-changing ground state solution.
    Moreover it  changes sign exactly once in $\mathbb R^3$.
\end{theo}

A function $f$ satisfying our assumptions is 
$$f(t) = \frac{t^{5}}{1+t^{2}}, \quad t\in \mathbb R,$$
 which has as primitive $$ F(t) = \frac{t^{4}}{4} - \frac{t^{2}}{2} + \frac{1}{2}\ln(1+t^{2}).$$
Clearly this function does not satisfy the assumption
$$\lim_{t\to \infty}\frac{F(t)}{t^{4}}=+\infty$$
required in \cite{ASS}.
Moreover, the case $f(t) = |t|^{p-1}t, p\in (3,5)$ studied in \cite{WZ} does not satisfies \eqref{f_{4}}.
So the present paper gives a new contribution in studying sign-changing solutions for the  Schr\"odinger-Poisson problem in the asymptotically cubic nonlinearity and can be seen as a counterpart of the papers
 \cite{ADP,ASS,WZ}.

\medskip

As we said before, we   use variational methods: the solution will be found
as the minimum of $I$,  in the context of radial functions,
 on the constraint already introduced in \cite{WZ}.
Nevertheless,  the main difficulty is to show that the constraint on which minimize
the functional is nonempty under our assumptions on $f$;
indeed all the techniques of the above cited
papers concerning the supercubic case (see also \cite{MMS} for the single equation)
 do not work and some new ideas have been necessary.

\medskip

The organisation of the paper is the following. 

In Section \ref{sec:preliminaries}
we recall and give some preliminary facts. 

In Section \ref{sec:useful}  the
set on which minimize the functional is introduced, and some of its properties proved.
Indeed this Section collects all the ingredients we need in order to prove the result.
In particular it is stated Proposition \ref{prop:M} which says that the
set on which we minimize is nonempty.

In Section \ref{sec:proof} the main result is proved. 

Finally in the Appendix we prove the  technical Proposition \ref{prop:M}.

\medskip

\subsection*{Notations}
We conclude this Introduction by introducing few basic notations. In all the paper, $H^{1}(\mathbb R^{3})$
is the usual Sobolev space with norm
$$\|u\| =\left( |\nabla u|_{2}^{2} +|u|_{2}^{2} \right)^{1/2},$$
where $|\cdot|_{p}$ is the usual $L^{p}-$norm in $\mathbb R^{3}$.
Moreover $H^{1}_{rad}(\mathbb R^{3})$ denotes the subspace of radial functions.
We need also the space $D^{1,2}(\mathbb R^{3})$, which is defined as the completion of the test functions
with respect to the norm $\|\cdot\|_{D}: =|\nabla \cdot|_{2}$. 
As usual, $D^{1,2}_{rad}(\mathbb R^{3})$ is the subspace of radial functions.
We denote with
$\norm{\, \cdot\, }_\star$  the norm for the space of continuous linear functionals defined on $H^{1}_{rad}(\mathbb R^{3})$. 
As customary, $meas(\cdot)$ takes the Lebesgue measure of a set. Furthermore, 
we use the letter $C$ to denote a positive constant 
whose value may change from line to line.
Other notations will be introduced as soon as we need.

\section{Preliminaries and known results}\label{sec:preliminaries}
We begin by recalling that in particular from \eqref{f_{1}}, \eqref{f_{3}} and \eqref{f_{4}}, given $\varepsilon >0$ and $q\in (4, 6)$, there exists $C_\varepsilon >0$  constant such that
\begin{equation}\label{III.4}
    \abs{f(t)}\leq \varepsilon \abs{t}+ C_\varepsilon \abs{t}^{q-1}\quad \text{and} \quad \abs{F(t)}\leq \frac{\varepsilon}{2}t^2+ \frac{C_\varepsilon}{q} \abs{t}^q, \quad \forall t\in \mathbb R.
\end{equation}
Moreover from \eqref{f_{5}} it follows that
\begin{equation}\label{eq:strictincreasing}
\text{ the function  }t\in [0,+\infty)\mapsto f(t)t-4F(t)\in \mathbb R \  \text{ is strictly increasing}
\end{equation} 
(see e.g. \cite[Lemma 2.3]{Liu}) and then by \eqref{f_{2}} and \eqref{eq:strictincreasing} we infer 
\begin{equation}\label{eq:positivity}
f(t)t-4F(t)\geq0, \quad \forall t \in \mathbb R.
\end{equation}

Let us recall also a result on the whole $\mathbb R^{3}$. This will have a major role in all our analysis. In the paper \cite{ADP} it was studied  the  problem
\begin{equation}\label{ADPproblem2}
    \left\{
        \begin{array}{ll}
            -\Delta u+\lambda \phi u=g(u) &\text{ in } \mathbb R^3 \\
            -\Delta \phi= \lambda u^2  &\text{ in } \mathbb R^3,
        \end{array}
    \right.
\end{equation}
under the following assumptions on $g$:
\begin{enumerate}[label=(g\arabic*),ref=g\arabic*,start=1]
\item
$g \in C (\mathbb R ,\mathbb R)$, \medskip
\item
$-\infty < \liminf_{t\rightarrow 0^+}g(t)/t\leq \limsup_{t\rightarrow 0^+}g(t)/t=-m<0$,\medskip
\item
$-\infty \leq \limsup_{t \rightarrow \infty} {g(t)/t^5}\leq 0$, \medskip
\item
there exists $\zeta>0$ such that $G(\zeta):=\int_0^{\zeta}g(t)dt>0$. \medskip
\end{enumerate}
In other words $g$ is a general nonlinearity satisfying the Berestycki-Lions assumptions. The authors prove that there is a $\lambda_{0}>0$ such that for every $\lambda\in(0,\lambda_{0})$ the problem \eqref{ADPproblem2} has a nontrivial positive and radial solution.

Our problem \eqref{ourproblem} can be written as \eqref{ADPproblem2} just renaming the nonlinearity and the potential $\phi$; then in virtue of \cite{ADP} we have the following
\begin{lem}\label{le:facile}
Under conditions \eqref{f_{1}}, \eqref{f_{3}} and \eqref{f_{4}} there is a $\lambda_{0}>0$ such that for every $\lambda\in(0,\lambda_{0})$ problem \eqref{ourproblem} has a positive and radial solution $\mathfrak u$.
\end{lem}

Recall also that  
if $\phi_{u}$ is the unique solution of the Poisson equation 
$$-\Delta \phi = u^{2} \quad \text{ in }\mathbb R^{3}$$
for a given $u\in H_{rad}^{1}(\mathbb R^{3})$, 
then 
\begin{equation}\label{eq:stimafacile}
\exists \, C>0 : 
\|\phi_{u}\|_{D}^{2} = \int_{\mathbb R^3} \phi_u u^2 dx\leq C\norm{u}^4.
\end{equation}
For this see e.g. \cite{R}.

The smallness of $\lambda$ stated in Lemma \ref{le:facile}
 is a necessary condition in order to have a positive
 solution for problem \eqref{ourproblem}. In our case, that is in presence of a nonlinearity which is not of power type, this can be seen in the following way. Let $u_\lambda \in H^1(\mathbb R^3)$ be a  positive  solution of \eqref{ourproblem} 
for a certain value of $\lambda>0$ and let  $\phi_{\lambda} := \phi_{u_{\lambda}}\in D^{1,2}(\mathbb R^{3})$. Multiplying the first equation of \eqref{ourproblem} by $u_\lambda$ and  integrating on $\mathbb R^3$ we obtain
\begin{equation}\label{eq:disug}
0<\norm{u_\lambda}^2 =\int_{\mathbb R^3} \left( \frac{f(u_\lambda)}{u_\lambda^3}-\lambda \frac{\phi_\lambda}{u_\lambda^2}\right) u_\lambda^4\, dx < \int_{\mathbb R^3} \left( 1-\lambda \frac{\phi_\lambda}{u_\lambda^2}\right)u_\lambda^4\, dx = \int_{\mathbb R^3} \left( u_\lambda^4-\lambda \phi_\lambda u_\lambda^2\right) dx 
\end{equation}
so that
\begin{equation}\label{III.1}
    \lambda  < \frac{\displaystyle{\int _{\mathbb R^3}u_\lambda^4dx}}{\displaystyle{\int _{\mathbb R^3} \phi_\lambda u_\lambda^2dx}}.
\end{equation}
We define the set made by scalar multiple of solutions, precisely
\[
    \mathcal S_{\lambda}:=\left\{ {\norm{\phi_\lambda}}_{D}^{-1/2}u_{\lambda}\in H^1(\mathbb R^3) \setminus\{0\}:  u_{\lambda} \text{ solves } \eqref{ourproblem} \right\}\neq\emptyset.
\]
Now if $v\in \mathcal S_{\lambda}$ then $v={\norm{\phi_\lambda}}_{D}^{-1/2}u_{\lambda}$ for some solution $u_{\lambda}$ and $\|\phi_{v}\|_{D}=1$. Then by \eqref{eq:stimafacile},
\[
    1=\norm{\phi_v}_{D}^{2}=\int _{\mathbb R^3} \abs{\nabla \phi_v}^2 dx =\int _{\mathbb R^3} \phi_v v^2 dx\leq C\norm{v}^4
\]
showing that $K :=\inf _{v\in \mathcal S_{\lambda}} \norm{v}^4>0.$ Moreover recalling \eqref{III.1}, 
\[
    \lambda  < \frac{{\norm{\phi_\lambda}}_{D}^{-2}\displaystyle{\int _{\mathbb R^3}u_{\lambda}^4dx}}{{\norm{\phi_\lambda}}_{D}^{-2}\displaystyle{\int _{\mathbb R^3} 
    \phi_\lambda u_{\lambda}^2dx}}=\frac{\displaystyle{\int _{\mathbb R^3}v^4dx}}{\displaystyle{\int _{\mathbb R^3} \phi_{v}v^2dx}}= \int _{\mathbb R^3}v^4dx\leq C \norm{v}^4
\]
and then ${\lambda }/{C}\leq K$ i.e. $\lambda \leq K C.$

\medskip

Recall also that if $u$ is radial, the unique solution $\phi_{u}$ of the Poisson equation 
is also radial, and by the Newton's Theorem it can be written as (we omit from now on the factor $1/4\pi$)
$$\phi_{u}(r)=\frac{1}{r}\int_0^\infty u^2(s)s\min\{s,r\}ds.$$
From this representation it is easy to see that $\phi_{u}$ is decreasing in the radial coordinate.
Indeed let $0<r_{1}<r_{2}$; we have that
\begin{align*}
    \phi(r_1)&=\frac{1}{r_1}\int_0^\infty u^2(s)s\min\{s,r_1\}ds\\
    &=\frac{1}{r_1}\int_0^{r_1}u^2(s)s^2ds+\int_{r_1}^\infty u^2(s)sds\\
    &=\frac{1}{r_1}\int_0^{r_1}u^2(s)s^2ds+\int_{r_1}^{r_2}u^2(s)sds+\int_{r_2}^\infty u^2(s)sds\\
    &>\frac{1}{r_2}\int_0^{r_1}u^2(s)s^2ds+\frac{1}{r_2}\int_{r_1}^{r_2}u^2(s)s^2ds+\int_{r_2}^\infty u^2(s)sds\\
    &=\frac{1}{r_2}\int_0^{r_2}u^2(s)s^2ds+\int_{r_2}^\infty u^2(s)sds\\
    &=\frac{1}{r_2}\int_0^\infty u^2(s)s\min\{s,r_2\}ds\\
    &=\phi(r_2).
\end{align*}
The radial  solutions of \eqref{ourproblem}
are critical points of the $C^1$ functional $I:H^1_{rad}(\mathbb R^3) \rightarrow \mathbb R$ defined 
in \eqref{eq:I};
indeed the derivative at $u$ has the expression
\[
    I'(u)[v]=\int_{\mathbb R^3} \left( \nabla u  \nabla v +uv\right) dx +\lambda\int_{\mathbb R^3} \phi_u uv dx- \int_{\mathbb R^3} f(u)vdx, \quad \text{for } v\in H^1_{rad}(\mathbb R^3)
\]
and clearly all the nontrivial critical points of $I$ belong to the so called   Nehari set
$$\mathcal N:=\left\{u\in H^1_{rad}(\mathbb R^3)\setminus \{0\}:I'(u)[u]=0\right\}.$$
Let us define the map
$$\gamma(u) := I'(u)[u]= \int_{\mathbb R^{3}}\left( |\nabla u|^{2} +u^{2}\right)dx+\lambda\int_{\mathbb R^{3}} \phi_{u} u^2dx-\int_{\mathbb R^{3}}f(u)udx,$$
evidently continuous.
Two basic properties of $\mathcal N$ are the following. Others will be shown in the next section.

First, $\mathcal N$ is bounded away from zero in $H^{1}_{rad}(\mathbb R^{3})$. Indeed from $u\not\equiv0$ and $\gamma(u)=0$ by \eqref{III.4} with $\varepsilon=1/2$ 
and $q\in(4,6)$ we have
    \begin{align*}
        \int_{\mathbb R^3} \left( \abs{\nabla u}^2 +u^2\right) dx &< \int_{\mathbb R^3} \left( \abs{\nabla u}^2 +u^2\right) dx+ \lambda\int_{\mathbb R^3} \phi_uu^2 dx\\
        &= \int_{\mathbb R^3} f(u)u dx\\
        & \leq \frac{1}{2} \int_{\mathbb R^3} u^2 dx +C_{1/2} \int_{\mathbb R^3} \abs{u}^q dx
    \end{align*}
    so that
    \begin{equation}\label{eq:Lq}
        \frac{1}{2}\|u\|^{2}< \int_{\mathbb R^3} \left(\abs{\nabla u}^2 +\frac{1}{2}u^2\right) dx < C_{1/2} \int_{\mathbb R^3} \abs{u}
        ^q dx<C \|u\|^{q}.
    \end{equation}
    
    
    \begin{rem}\label{rem:less}
    Of course the above computation also holds  if $\gamma(u)<0$. Moreover we can 
    deduce from \eqref{eq:Lq} and the Sobolev embeddings,   that there exists $L>0$ such that
    $$ u\in \mathcal N,\quad\text{or,}\quad u\in H^1_{rad}(\mathbb R^3)\text{ with }\gamma(u)<0\quad\Longrightarrow\quad 0<L< |u|_{q} , \|u \| .$$
    \end{rem}

Second, the functional $I$ is bounded from below on $\mathcal N$; indeed, for $u\in \mathcal N$,
\begin{align*}
        I(u)&=I(u)-\frac{1}{4}I'(u)[u]\\
        &=\frac{1}{4}\norm{u}^2-\int _{\mathbb R^3} F(u)dx +\frac{1}{4}\int _{\mathbb R^3}f(u)u dx\\
        & > \int_{\mathbb R^3} \left[ \frac{1}{4}f(u)u-F(u) \right] dx\\
        &\geq0. \numberthis \label{eq:Ibb}
    \end{align*}
in virtue of \eqref{eq:positivity}.
Through the paper we will use repeatedly inequalities like
$$I(u)>\int_{\mathbb R^3} \left[ \frac{1}{4}f(u)u-F(u) \right] dx, \ \ \text{ whenever } \  u\in \mathcal N.$$


\begin{rem}\label{convergences} Recall also the following.
    Given a sequence $\{u_n\}\subset H^1_{rad}(\mathbb R^3)$ with $u_n\rightharpoonup u$ in $H^1_{rad}(\mathbb R^3)$, since $H^1_{rad}(\mathbb R^3)\hookrightarrow\hookrightarrow L^p(\mathbb R^3)$ for $2<p<6$, we have 
    $$\int_{\mathbb R^3}f(u_n)u_ndx\rightarrow\int_{\mathbb R^3}f(u)udx\quad\text{and}\quad\int_{\mathbb R^3}F(u_n)dx\rightarrow\int_{\mathbb R^3}F(u)dx,$$
    and  \emph{(see \cite[Lemma 2.1]{R})}
    $$\phi_{u_n}\rightarrow\phi_u\quad\text{in }D_{rad}^{1,2}(\mathbb R^3).$$
    Even more, if $\{v_n\}$ is another sequence in $H^1_{rad}(\mathbb R^3)$ with $v_n\rightharpoonup v$ in $H^1_{rad}(\mathbb R^3)$, then
    $$\int_{\mathbb R^3}\phi_{u_n}v_n^2dx\rightarrow\int_{\mathbb R^3}\phi_uv^2dx,$$
    because $\phi_{u_n}\rightarrow\phi_u$ in $L^6(\mathbb R^3)$ and $v_n^2\rightarrow v^2$ in $L^{6/5}(\mathbb R^3)$.
\end{rem}

\medskip

\section{Some useful results}\label{sec:useful}
In this Section, by following \cite{WZ} we introduce the set
on which minimize the functional $I$ and give all the properties we need to work with.

Let us denote hereafter, $u^+(x):=\max\{u(x),0\}$ and $u^-(x):=\min\{u(x),0\}$. In this way we have the 
decomposition $u=u^{+}+u^{-}$.

Now, although $\phi_{u} = \phi_{u^{+}}+\phi_{u^{-}}$ (see \cite{WZ})
it results
\begin{eqnarray*}
I(u)= I(u^{+}) + I(u^{-}) +\frac{\lambda}{2}\int_{\mathbb R^{3}}\phi_{u^{-}} (u^{+})^{2} dx
\end{eqnarray*}
then an extra term, with respect to the case $\lambda=0$, appears in the decomposition of $I$. 
We  use here, and throughout the paper, that
$$\forall u,v\in H^{1}(\mathbb R^{3}) : \int_{\mathbb R^{3}} \phi_{u} v^{2}dx = \int_{\mathbb R^{3}} \phi_{v} u^{2}dx.$$
Moreover, if $u^{\pm}\not\equiv0$,
\begin{eqnarray*}
I'(u)[u^{+}] &=& I'(u^{+})[u^{+}]+\lambda\int_{\mathbb R^{3}}\phi_{u^{-}}(u^{+})^{2}> I'(u^{+})[u^{+}],\\
I'(u)[u^{-}] &=& I'(u^{-})[u^{-}]+\lambda\int_{\mathbb R^{3}}\phi_{u^{+}}(u^{-})^{2}> I'(u^{-})[u^{-}],
\end{eqnarray*}
and this implies that, if $u$ is a sign changing solution, 
hence in particular $u\in\mathcal M$,
then $ I'(u^{+})[u^{+}],  I'(u^{-})[u^{-}]<0$ and so
$$u^{\pm}\not\in \mathcal N.$$
For this reason $\mathcal N$ is not a good set on which find the critical points
of $I$; a natural choice of the set on which study the functional $I$ is
$$\mathcal M:=\left\{u\in\mathcal N:I'(u)[u^+]=0;\ u^{\pm}\not\equiv0\right\}.$$
Observe that if $u\in \mathcal M$ then $I'(u)[u^-]=0$ and that  any sign-changing solution is on $\mathcal M$.
Moreover $\mathcal M$ is not a smooth manifold.

By the above decomposition we have, for $\alpha,\beta>0$ and $u\in H^{1}(\mathbb R^{3})$:
\begin{equation}\label{eq:decomposI'}
I'(\alpha u^{+} + \beta u^{-})[\alpha u^{+}] = I'(\alpha u^{+})[\alpha u^{+}] + \alpha^{2} \beta^{2}\lambda\int_{\mathbb R^{3}} \phi_{ u^{+}} ( u^{-})^{2}dx
\end{equation}
that will be useful later on.

The main results of this Section  are the following: 
\begin{itemize}
\item $\mathcal M$ is not empty (see Proposition \ref{prop:M}), \smallskip
\item if $\inf_{\mathcal M}I$ is achieved, it is a critical level  (see Proposition \ref{attained}).
\end{itemize}

Note that the second item is not immediate since $\mathcal M$
is not a smooth manifold, and indeed a preliminary work is necessary.

\begin{prop}\label{prop:M}
For $\lambda$ sufficiently small, the set $\mathcal M$ is not empty.
\end{prop}
The proof of this fact is postponed to the Appendix being somehow technical. It strongly uses
the solution $\mathfrak u$ found in Lemma \ref{le:facile} for $\lambda$ small.
In virtue of this, from now on the parameter  $\lambda$ has to be considered  fixed in $(0, \lambda_{0})$.

\medskip

If we define the maps
$$\gamma_{\pm}: u\in H^{1}_{rad}(\mathbb R^{3}) \longmapsto I'(u)[u^{\pm}] \in \mathbb R$$
we can write
$$\mathcal M = \left\{ u\in \mathcal N : \gamma_{+}(u)= 0, u^{\pm}\not\equiv 0\right\}.$$
By the continuity of 
$u\in H^1_{rad}(\mathbb R^3)\mapsto u^+\in H^1_{rad}(\mathbb R^3)$, see \cite[Proposition 7.2]{CR}, 
we deduce that $\gamma_{\pm}$ are continuous.
By the definitions it also holds
\begin{equation}\label{eq:decomp+}
\gamma_{+}(u) = I'(u^{+})[u^{+}]+\lambda\int_{\mathbb R^{3}} \phi_{u^{+}} (u^{-})^{2}dx=\gamma(u^{+})+\lambda\int_{\mathbb R^{3}} \phi_{u^{+}} (u^{-})^{2}dx
\end{equation}
and
\begin{equation}\label{eq:decomp-}
\gamma_{-}(u) = I'(u^{-})[u^{-}]+\lambda\int_{\mathbb R^{3}} \phi_{u^{+}} (u^{-})^{2}dx=\gamma(u^{-})+\lambda\int_{\mathbb R^{3}} \phi_{u^{+}} (u^{-})^{2}dx.
\end{equation}

Due to the continuity of the maps defined above, we have the following useful result.
\begin{prop}
The set $\mathcal M$ is closed.
\end{prop}
\begin{proof}
We know that
$$\mathcal M=\gamma^{-1}(\{0\})\cap\gamma_+^{-1}(\{0\})\cap\left\{v\in H^1_{rad}(\mathbb R^3):v^\pm\not\equiv0\right\},$$
and
    \begin{equation*}\label{closed1}
        \gamma^{-1}(\{0\}),  \ \gamma_+^{-1}(\{0\})\quad\text{are closed sets in } H^1_{rad}(\mathbb R^3).
    \end{equation*}
    
    Let us consider $\{u_{n}\}\subset \mathcal M$ such that $u_{n}\to u$ in $H^{1}_{rad}(\mathbb R^{3})$. 
   In particular 
   \begin{equation}\label{eq:primaparte}
   u\in \gamma^{-1}(\{0\})\cap \gamma_+^{-1}(\{0\}).
   \end{equation}     
    From $ \gamma_{+}(u_{n})=0$ and \eqref{eq:decomp+}
    we have
$$
    \gamma(u_{n}^{+}) =      \gamma(u_{n}^{+})  - \gamma_{+}(u_{n})= - \lambda\int_{\mathbb R^{3}} \phi_{u_{n}^{-}} (u_{n}^{+})^{2}
    <0
$$
    and then by Remark \ref{rem:less} it is $|u_{n}^{+}|_{q}> L>0$ with $q\in(4,6)$.
    The sequence $\{u_n^+\}$ is bounded in $H^1_{rad}(\mathbb R^3)$, being $\norm{u_n^+}\leq\norm{u_n}$, and then by the compact embedding $H^1_{rad}(\mathbb R^3)\hookrightarrow\hookrightarrow L^q(\mathbb R^3)$ we get the existence of some function $v$ such that
        \begin{equation}\label{eq:+}
u_n^+\rightarrow v\geq0 \ \text{ in } \  L^q(\mathbb R^3) \ \text{ and }\ |v|_{q}\geq L>0.
    \end{equation}
  Since it is also 
  $$\mathcal M=\gamma^{-1}(\{0\})\cap\gamma_-^{-1}(\{0\})\cap\left\{v\in H^1_{rad}(\mathbb R^3):v^\pm\not\equiv0\right\},$$
  we can argue as before by using $\gamma_{-}, u_{n}^{-}$ and \eqref{eq:decomp-}. In such a way we arrive at
    \begin{equation}\label{eq:-}
    u_n^-\rightarrow w\leq0 \ \text{ in } \  L^q(\mathbb R^3) \ \text{ and }\ |w|_{q}\geq L>0.
    \end{equation}
    But since $u_{n}=u_{n}^{+} + u_{n}^{-} \to u^{+}+u^{-}$ in $L^{q}(\mathbb R^{3})$ it has to be 
    $$u^{+} = v, \ \ u^{-} = w$$ and by \eqref{eq:+} and \eqref{eq:-} we infer 
    \begin{equation}\label{eq:secondaparte}
    u\in\left\{v\in H^1_{rad}(\mathbb R^3):v^\pm\not\equiv0\right\}.
    \end{equation}
    Then \eqref{eq:primaparte} and \eqref{eq:secondaparte} give the conclusion.
    \end{proof}

Also in the asymptotically cubic case there is the unicity of the projection on the Nehari set,
and on the fibers the functional $I$ achieves the maximum on $\mathcal N$.
More precisely we have the following:
\begin{lem}\label{domination}
If $u\in\mathcal N$, then
\begin{enumerate}[label=(\roman*),ref=\roman*]
\item\label{i1} $I(tu)<I(u)$  for every $t> 0,\ t\neq 1$;
\medskip

\item\label{i2}  $ tu\notin\mathcal N$ for every $ t>0,\ t\neq 1.$ More specifically,
$$\gamma(tu)>0 \ \  \text{ for } t\in(0,1)\quad \text{ and  }\quad 
   \gamma(tu)<0 \ \  \text{ for } t>1.$$
\end{enumerate}

    \end{lem}
\begin{proof}
    To show \eqref{i1}, let $u\in\mathcal N$ and define  the function 
    \[
        \xi(t):=\left(\frac{t^2}{2}-\frac{t^4}{4} \right) \left( \abs{\nabla u}^2+ u^2 \right)+\frac{t^4}{4}f(u)u-F(tu),\quad t\geq 0
    \]
    on the set $\{x\in \mathbb R^3: u(x)\neq0\}$ which has positive measure. For any $t>0$
    \[
        \xi '(t)=t(1-t^2) \left( \abs{\nabla u}^2+ u^2 \right)+t^3 u^4\left[ \frac{f(u)}{u^3}-\frac{f(tu)}{t^3u^3} \right],
    \]
    and thus, by \eqref{f_{5}}, $\xi(t)  < \xi(1)$, for every $t>0$, $t\neq 1$. Then, after integration on $\mathbb R^3$ and using that $\gamma(u)=I'(u)[u]=0$, we have
    \begin{align*}
        I(tu)&= \left( \frac{t^2}{2}-\frac{t^4}{4}\right) \norm{u}^2 +\int_{\mathbb R^3}\left[ \frac{t^4}{4}f(u)u-F(tu)\right] dx\\
        & < \frac{1}{4}\norm{u}^2 +\int_{\mathbb R^3}\left[ \frac{1}{4}f(u)u-F(u)\right] dx\\
        &= I(u).
    \end{align*}
    To prove \eqref{i2}, simply observe that by \eqref{f_{5}}, in case $0<t<1$ we have 
    \[
        \gamma(tu)=I'(tu)[tu]> t^4\left[\norm{u}^2+\lambda \int_{\mathbb R^3}\phi_u u^2dx -\int_{\mathbb R^3}\frac{f(tu)}{t^3}u\, dx\right] > t^4 \gamma(u)= 0,
    \]
    while in case $t>1$ we have
    \[
        \gamma(tu)=I'(tu)[tu]< t^4\left[\norm{u}^2+\lambda \int_{\mathbb R^3}\phi_u u^2dx -\int_{\mathbb R^3}\frac{f(tu)}{t^3}u\, dx\right] < t^4 \gamma(u)= 0
    \]
and the proof is completed.
\end{proof}
A similar assert can be done for the elements of $\mathcal M$.
\begin{lem}\label{domination1}
   If $u\in\mathcal M$, then
   \begin{enumerate}[label=(\roman*),ref=\roman*]
   \item \label{M1}$I(su^++tu^-)<I(u)$ for $s,t\geq0,\ s\neq1\text{ or }t\neq1$; \medskip
   \item\label{M2}$  tu\notin\mathcal M$ for every $t>0,\ t\neq 1$. More specifically,
   $$\gamma_{\pm}(tu)>0 \ \  \text{ for } t\in(0,1)\quad \text{ and  }\quad 
   \gamma_{\pm}(tu)<0 \ \  \text{ for } t>1.$$
   \end{enumerate}

\end{lem}
  
    \begin{proof}
    If $u\in\mathcal M$, then $\gamma(u)=I'(u)[u]=\gamma_{\pm}(u)=I'(u)[u^{\pm}]=0$.
By making use  of \eqref{eq:decomposI'}-\eqref{eq:decomp-} 
    we get
    \begin{eqnarray}
        I(su^++tu^-)&=&I(su^++tu^-)-\frac{s^4}{4}I'(u)[u^+]-\frac{t^4}{4}I'(u)[u^-] \nonumber\\
        &=&\left(I(su^+)-\frac{s^4}{4}I'(u^+)[u^+]\right)+\left(I(tu^-)-\frac{t^4}{4}I'(u^-)[u^-]\right) \nonumber \\
        &&-\frac{\lambda}{4}(s^2-t^2)^2\int_{\mathbb R^3}\phi_{u^+}(u^-)^2dx \nonumber \\
        &<& A+B \label{eq:AB}
    \end{eqnarray}
    where
    $$A:= I(su^+)-\frac{s^4}{4}I'(u^+)[u^+], \quad B:=I(tu^-)-\frac{t^4}{4}I'(u^-)[u^-].$$
    
    Explicitly
    $$A=\left(\frac{s^2}{2}-\frac{s^4}{4}\right)\norm{u^+}^2+\int_{\mathbb R^3}\left[\frac{s^4}{4}f(u^+)u^+-   F(su^+)\right]dx$$
    so that, arguing   as in Lemma \ref{domination} by using the function
    $$\xi_+(s):=\left(\frac{s^2}{2}-\frac{s^4}{4} \right) \left[ \abs{\nabla u^+}^2+ (u^+)^2 \right]+\frac{s^4}{4}f(u^+)u^+-F(su^+),$$
    we obtain $I(su^+)<I(u^+)$ and then
    $$A = I(su^+)-\frac{s^4}{4}I'(u^+)[u^+]<I(u^+)-\frac{1}{4}I'(u^+)[u^+],\quad\text{for }s>0,\ s\neq1.$$
    Similarly we have
    $$B=I(tu^-)-\frac{t^4}{4}I'(u^-)[u^-]<I(u^-)-\frac{1}{4}I'(u^-)[u^-],\quad\text{for }t>0,\ t\neq1,$$
    and consequently by \eqref{eq:AB} we infer
    \begin{align*}
        I(su^++tu^-)&<\left(I(u^+)-\frac{1}{4}I'(u^+)[u^+]\right)+\left(I(u^-)-\frac{1}{4}I'(u^-)[u^-]\right)\\
        &=\left(I(u^+)-\frac{1}{4}I'(u)[u^+]\right)+\left(I(u^-)-\frac{1}{4}I'(u)[u^-]\right)+\frac{\lambda}{2}\int_{\mathbb R^3}\phi_{u^+}(u^-)^2dx\\
        &=I(u^+)+I(u^-)+\frac{\lambda}{2}\int_{\mathbb R^3}\phi_{u^+}(u^-)^2dx\\
        &=I(u),
    \end{align*}
and \eqref{M1} follows. 

For \eqref{M2}, first note that
    \begin{align*}
        \gamma_+(tu)&=I'(tu)[tu^+] \\
        &=I'(tu^+)[tu^+]+\lambda\int_{\mathbb R^3}\phi_{tu^+}(tu^-)^2dx\\
        &=t^2\norm{u^+}^2+\lambda t^4\int_{\mathbb R^3}\phi_{u^+}u^2dx-\int_{\mathbb R^3}f(tu^+)tu^+dx.
    \end{align*}
    Since $\gamma_+(u)=0$, by \eqref{f_{5}}, in case $0<t<1$ 
    \[
        \gamma_+(tu)> t^4\left[\norm{u^+}^2+\lambda \int_{\mathbb R^3}\phi_{u^+}u^2dx -\int_{\mathbb R^3}\frac{f(tu^+)}{t^3}u^+dx\right] > t^4 \gamma_+(u)= 0,
    \]
    while in case $t>1$
    \[
        \gamma_+(tu)< t^4\left[\norm{u^+}^2+\lambda \int_{\mathbb R^3}\phi_{u^+}u^2dx -\int_{\mathbb R^3}\frac{f(tu^+)}{t^3}u^+dx\right] < t^4 \gamma_+(u)= 0.
    \]
    Furthermore from $\gamma(u)=I'(u)[u]=0$ and $\gamma_+(u)=0$, we have $\gamma_-(u)=I'(u)[u^-]=0$, and similar computations, together with \eqref{f_{2}}, give us
    $$\gamma_-(tu)>0,\text{ when }0<t<1,\quad\text{ and }\gamma_-(tu)<0,\text{ when }t>1$$
concluding the proof.
\end{proof}

\begin{rem}\label{domination2}
By using  the same ideas of Lemma \ref{domination1}  we can establish  the following  generalization. 
    
    Assume that $u\in\mathcal N$ is such that
    \begin{enumerate}
        \item $u=u_1+u_2+u_3$, with $\supp(u_i)\cap\supp(u_j)=\emptyset$, for $i<j$ where $i,j=1,2,3$; \medskip
        \item $u_i\not\equiv0$, for every $i=1,2,3$; \medskip
        \item $I'(u)[u_i]=0$, for every $i=1,2,3$. \medskip
    \end{enumerate}
    Then we have
    $$I(t_1u_1+t_2u_2+t_3u_3)<I(u),\quad\text{for }t_i\geq0,\ i=1,2,3,\text{ with at least one }t_i\neq1.$$
    Indeed, denoting with $\phi_i:=\phi_{u_i}$ for $i=1,2,3$,  under our assumptions
    \begin{align*}
        I(t_1u_1+t_2u_2+t_3u_3)&=\sum_{i=1}^3I(t_iu_i)+\frac{\lambda}{2}\sum_{i<j}t_{i}^{2}t_{j}^{2}\int_{\mathbb R^3}\phi_iu_j^2dx\\
        I'(u)[u_i]&=I'(u_i)[u_i]+\lambda\sum_{j\neq i}\int_{\mathbb R^3}\phi_ju_i^2dx,
    \end{align*}
    so that
    \begin{align*}
        I(t_1u_1+t_2u_2+t_3u_3)&=I(t_1u_1+t_2u_2+t_3u_3)-\sum_{i=1}^3\frac{t_i^4}{4}I'(u)[u_i]\\
        &=\sum_{i=1}^3\left(I(t_iu_i)-\frac{t_i^4}{4}I'(u_i)[u_i]\right)-\frac{\lambda}{4}\sum_{i<j}(t_i^2-t_j^2)^2\int_{\mathbb R^3}\phi_iu_j^2dx.
    \end{align*}
    By defining the functions
    $$\xi_i(t_i):=\left(\frac{t_i^2}{2}-\frac{t_i^4}{4} \right) \left[ \abs{\nabla u_i}^2+ (u_i)^2 \right]+\frac{t_i^4}{4}f(u_i)u_i-F(t_iu_i),\quad i=1,2,3,$$
    and arguing as in Lemma \ref{domination} and Lemma \ref{domination1} we have that
    $$I(t_iu_i)-\frac{t_i^4}{4}I'(u_i)[u_i]<I(u_i)-\frac{1}{4}I'(u_i)[u_i],\quad\text{ for every }t_i>0,\ t_i\neq1.$$
    Thus
    \begin{align*}
        I(t_1u_1+t_2u_2+t_3u_3)&<\sum_{i=1}^3\left(I(u_i)-\frac{1}{4}I'(u_i)[u_i]\right)\\
        &=\sum_{i=1}^3\left(I(u_i)-\frac{1}{4}I'(u)[u_i]\right)+\frac{\lambda}{2}\sum_{i<j}\int_{\mathbb R^3}\phi_ju_i^2dx\\
        &=\sum_{i=1}^3I(u_i)+\frac{\lambda}{2}\sum_{i<j}\int_{\mathbb R^3}\phi_ju_i^2dx\\
        &=I(u_1+u_2+u_3)\\
        &=I(u),
    \end{align*}
    and we are done.
\end{rem}

The next result will be useful for dealing with minimising sequences.
\begin{prop}\label{boundedsequence}
    Let $\{u_n\}\subset\mathcal M$ be a sequence such that $\{I(u_n)\}$ is  bounded. Then $\{u_n\}$ is bounded.
\end{prop}
\begin{proof}
    Assume that the statement does not hold, i.e., there exists a subsequence again denoted by $\{ u_n\}$ such that $\{I(u_n)\}$ is bounded but $\norm{u_n}\rightarrow \infty$ when $n\rightarrow \infty$. Then, up to a subsequence, $I(u_n)\rightarrow l\geq0$ by \eqref{eq:Ibb}.
        
    If $l>0$, we define $v_n:=2\sqrt l\,  u_n/\norm{u_n}$, so that $\norm{v_n}=2\sqrt l$. If $l=0$, we first define $t_n:=1/\norm{u_n}$ and
    then $v_n:=t_n u_n$, so that $\norm{v_n}=1$.
    
    \medskip
    
    {\bf Claim: }  there exist numbers $r,d>0$ and a sequence $\{y_n\} \subset \mathbb R^3$ such that
    \begin{equation}\label{claim}
        \liminf _{n\rightarrow \infty} \int _{B_r(y_n)} v_n^{2}\, dx \geq d>0.
    \end{equation}
    If the claim does not hold, then by Lions' lemma, $v_n \rightarrow 0$ in $L^p(\mathbb R^3)$, for $2<p<6$. Hence, in particular for $q\in(4,6)$, using \eqref{III.4} we obtain
    \[
        \bigg\lvert \int _{\mathbb R^3} F(v_n) dx \biggr\rvert \leq \frac{\varepsilon}{2} \int _{\mathbb R^3}v_n^{\, 2}\, dx+ \frac{C_\varepsilon}{q} \int _{\mathbb R^3} \abs{v _n}^q\, dx,
    \]
    and thus we have
    \[
        \lim _{n\rightarrow \infty} \int _{\mathbb R^3} F(v_n)dx=0.
    \]
    Therefore, when $l>0$
    \[
        I(v_n)=\frac{1}{2}\norm{v_n}^2 +\frac{\lambda}{4}\int _{\mathbb R^3} \phi_{v_n}v_n^{\, 2}\, dx -\int _{\mathbb R^3} F(v_n)dx=2l+\frac{\lambda}{4}\int_{\mathbb R^3}\phi_{v_n} v_n^2dx+o_n(1),
    \]
    while when $l=0$
    \begin{equation}\label{case0}
        \liminf_{n\rightarrow \infty} I(v_n)\geq \liminf_{n\rightarrow \infty} \left[ \frac{1}{2}\norm{v_n}^2 -\int_{\mathbb R^3} F(v_n)dx \right]=\frac{1}{2}.
    \end{equation}
    On the other hand, since $I'(u_n)[u_n]=0$, using Lemma \ref{domination} we get
    \[
        I(tu_n)\leq I(u_n),\quad \forall t>0.
    \]
    Taking $t=t_n:=2\sqrt l/\norm{u_n}$ when $l>0$, and $t=t_n:=1/\norm{u_n}$ when $l=0$, we have
    \[
        I(v_n)\leq I(u_n)=l+o_n(1).
    \]
    This fact implies, when $l=0$, that $I(v_n)\leq I(u_n)=o_n(1)$, contrary to \eqref{case0}, while when $l>0$,
    \[
        2l+o_n(1)< 2l+\frac{\lambda}{4}\int_{\mathbb R^3}\phi_{v_n} v_n^2dx+o_n(1) =I(v_n)\leq I(u_n)=l+o_n(1),
    \]
    or
    \[
      0<l< o_n(1),
    \]
    an absurd. Therefore the Claim holds.
    
   \medskip
   
     However by the Claim we infer a contradiction in both cases: when $\{y_{n}\}$ is bounded or unbounded.
     This will complete the proof.
     
         Hereafter the arguments can be used indistinctly for $l\geq 0$ real number. 

\medskip

    \textbf{Case 1: $\{y_n\}$ is bounded}\\
    Then there is $\widehat r>0$ with $\{y_n\} \subset B_{\widehat r}$. By \eqref{claim}, eventually
    $$\int_{B_r(y_n)}v_n^2dx>\frac{d}{2}.$$
    Thus we can choose $\widetilde r>r+\widehat r$; for this radius it holds $B_r(y_n)\subset B_{\widetilde r}$ and hence
    $$\int_{B_{\widetilde r}}v_n^2dx>\frac{d}{2}.$$
    Since $\norm{v_n}$ is constant, there exists a subsequence still denoted by $\{v_n\}$ such that $v_n\rightharpoonup v$ in $H^1_{rad}(\mathbb R^3)$, from which $v_n\rightarrow v$ in $L^p_{loc}(\mathbb R^3)$ for $1\leq p<6$, and $v_n(x)\rightarrow v(x)$ a.e. $x\in\mathbb R^3$. In particular
    $$\int_{B_{\widetilde r}}v_n^2dx\rightarrow\int_{B_{\widetilde r}}v^2dx,\quad\text{and thus}\quad\int_{B_{\widetilde r}}v^2dx\geq\frac{d}{2}>0,$$
    implying that $v\not\equiv0$. But then there is $\Lambda \subset B_{\widetilde r}$ with $meas (\Lambda)>0$ such that $v(x)\neq 0$, for every $x\in\Lambda$. Thence for $x\in\Lambda$ fixed and a constant $k>0$, $v_n(x)=ku_n(x)/\norm{u_n}\neq0$ for every $n$ large enough, so we can claim that $u_n(x)\neq 0$ eventually. As a consequence of $\norm{u_n}\rightarrow\infty$, $\abs{u_n(x)}\rightarrow\infty$. 
    Hence $\abs{u_n(x)}\rightarrow\infty$ for every $x\in\Lambda$. Since 
    $$I(u_n)\geq \int_{\mathbb R^3} \left[ \frac{1}{4}f(u_n)u_n-F(u_n) \right] dx\geq \int_{\Lambda}\left[ \frac{1}{4}f(u_n)u_n-F(u_n) \right] dx,$$
    by \eqref{f_{6}} and Fatou's lemma
    \[
        \liminf_{n\rightarrow \infty} I(u_n)\geq\int_{\Lambda}\liminf_{n\rightarrow \infty}\left[\frac{1}{4}f(u_n)u_n-F(u_n)\right] dx=\infty,
    \]
    from which $I(u_n)\rightarrow \infty$, contradicting that $I(u_n)\rightarrow l\in \mathbb R$.
    
    \medskip
    
    \textbf{Case 2: $\{y_n\}$ is unbounded}
    
    In this case we define a new sequence $\widehat v_n:=v_n(\cdot +y_n)$; note that $\norm{\widehat v_n}=\norm{v_n}$ is constant. Thus, up to a subsequence $\widehat v_n\rightharpoonup\widehat v$ in $H^1_{rad}(\mathbb R^3)$, so that $\widehat v_n\rightarrow\widehat v$ in $L^p_{loc}(\mathbb R^3)$ for $1\leq p<6$, and $\widehat v_n(x)\rightarrow\widehat v(x)$ a.e. $x\in\mathbb R^3$. From \eqref{claim}
    $$\liminf_{n\rightarrow\infty}\int_{B_r(y_n)}v_n^2dx=\liminf_{n\rightarrow\infty}\int_{B_r}\widehat v_n^2dx\geq d$$
    and hence
    $$\int_{B_r}\widehat v^2dx\geq d>0,$$
    implying that $\widehat v\not\equiv0$. But then there is $\Lambda \subset B_{r}$ with $meas (\Lambda) >0$ such that $\widehat v(x)\neq0$, for every $x\in\Lambda$. Thence for $x\in\Lambda$ fixed and a constant $k>0$, $\widehat v_n(x)=k u_n(x+y_n)/\norm{u_n}\neq0$ for every $n$ large enough, so that we can claim that $u_n(x+y_n)\neq0$ eventually. As a consequence of $\norm{u_n}\rightarrow\infty$, $\abs{u_n(x+y_n)}\rightarrow\infty$. Hence, as before, $\abs{u_n(x+y_n)}\rightarrow\infty$ for every $x\in\Lambda$. Therefore we have
    \begin{align*}
        I(u_n)&\geq\int_{\mathbb R^3}\left[\frac{1}{4}f(u_n)u_n-F(u_n)\right]dx\\
        &\geq\int_{B_r(y_n)}\left[\frac{1}{4}f(u_n)u_n-F(u_n)\right]dx\\
        &=\int_{B_r}\left[\frac{1}{4}f(u_n(x+y_n))u_n(x+y_n)-F(u_n(x+y_n))\right]dx\\
        &\geq\int_{\Lambda}\left[\frac{1}{4}f(u_n(x+y_n))u_n(x+y_n)-F(u_n(x+y_n))\right]dx,
    \end{align*}
    and by \eqref{f_{6}} and Fatou's lemma
    \[
        \liminf_{n\rightarrow \infty} I(u_n)\geq  \int_{\Lambda}\liminf_{n\rightarrow \infty} \left[\frac{1}{4}f(u_n(x+y_n))u_n(x+y_n)-F(u_n(x+y_n))\right] dx=\infty,
    \]
    from which $I(u_n)\rightarrow \infty$, which is again a contradiction.  
\end{proof}
The above proposition allows us to get the next result.

\begin{coro}
    We have $c:=\inf_{\mathcal M}I >0$.
\end{coro}
\begin{proof}
By \eqref{eq:Ibb} we know that $c\geq 0$. The proof is by contradiction. Assume that there is a sequence $\{u_n\}\subset\mathcal M$ with $I(u_n)\rightarrow 0$. By Proposition \ref{boundedsequence}, $\{u_n\}$ is bounded; let $\overline u \in H^1_{rad}(\mathbb R^3)$ such that, up to a subsequence, $u_n\rightharpoonup \overline u$ in $H^1_{rad}(\mathbb R^3)$ and $u_n(x)\rightarrow \overline u(x)$ a.e. in $\mathbb R^{3}$.
    Using Fatou's lemma and \eqref{eq:positivity}
    \begin{align*}
        0&=\liminf_{n\rightarrow \infty} I(u_n)\\
        &= \liminf_{n\rightarrow \infty}\left(I(u_n)-\frac{1}{4}I'(u_n)[u_n]\right)\\
        &\geq\liminf_{n\rightarrow \infty}\frac{1}{4}\norm{u_n}^2+\liminf_{n\rightarrow \infty}\int_{\mathbb R^3}\left[\frac{1}{4}f(u_n)u_n-F(u_n)\right]dx\\
        &\geq \frac{1}{4}\norm{\overline u}^2+\int_{\mathbb R^3}\left[\frac{1}{4}f(\overline u)\overline u-F(\overline u)\right] dx\\
        &\geq\frac{1}{4}\norm{\overline u}^2,
    \end{align*}
    from which $\norm{\overline u}=0$, which is a contradiction since $\mathcal N$ (and hence $\mathcal M$) is bounded away from zero by \eqref{eq:Lq}.
\end{proof}

Recall that $\mathcal M$ is not a smooth manifold, hence  the next  result will be fundamental.
\begin{prop}\label{attained}
    If the infimum $c$ of $I$ on $\mathcal M$ is attained, $c$ is a critical value of $I$.
\end{prop}
\begin{proof}
    Let be $u\in\mathcal M$ such that $I(u)=c=\inf_{\mathcal M}I$. We need  to prove that $I'(u)=0$. By contradiction assume that $I'(u)\neq 0$. Since $I$ is a $C^1$ functional on $H^1_{rad}(\mathbb R^3)$, there are $\delta >0$ and $\varepsilon _0>0$ such that
    \begin{equation}\label{nonullderivative}
\forall  v\in H^1_{rad}(\mathbb R^3) \text{ with } \norm{v-u}\leq 2\delta, \text{ we have } \norm{I'(v)}_\star \geq \varepsilon _0.
    \end{equation}
    Since $u\in \mathcal M$, we know by Remark \ref{rem:less} there is $L>0$ such that $\norm{u^+},\norm{u^-}> L>0$ and we can assume that $6\delta<L$.
    \vspace{3mm}
    
    On the set $Q:=\left[ \frac{1}{2}, \frac{3}{2}\right] \times \left[ \frac{1}{2}, \frac{3}{2}\right]$ we define the function 
    $$h:(\alpha,\beta)\in Q\longmapsto  \alpha u^+ + \beta u^-\in H^1_{rad}(\mathbb R^3).$$ Since
$        I'(u)[u^{\pm}]=0$,    by Lemma \ref{domination1}, $I(su^++tu^-)<I(u)$, for any $s,t>0$, with $s\neq1$ or $t\neq1$, and then
    \begin{equation}\label{levelc}
        I(h(\alpha,\beta))=I(\alpha u^++\beta u^-) <I(u)=c,
    \end{equation}
    for every pair $(\alpha,\beta)\in Q$ with $\alpha \neq 1$ or $\beta \neq 1$. Therefore
    \[
    c_0:=\max_{\partial Q} I\circ h < c.
    \]
    Defining $\varepsilon :=\min \{ (c-c_0)/2, \delta \varepsilon _0/8\}$, by \eqref{nonullderivative} we have
    \[
        v\in I^{-1}\left( [c-2\varepsilon, c+2\varepsilon ]\right) \cap B_{2\delta} (u)\quad \Longrightarrow \quad \norm{I'(v)}_\star \geq \varepsilon _0.
    \]
    But then, by the deformation lemma (see \cite[Lemma 2.3]{W}) there exists a deformation $\eta \in C([0,1]\times H^1_{rad}(\mathbb R^3),H^1_{rad}(\mathbb R^3))$ such that
   \begin{enumerate}[label=(\roman*),ref=\roman*]
        \item \label{DL:i}$\eta(t,v)=v$, if $t=0$ or $v\notin I^{-1}\left( [c-2\varepsilon, c+2\varepsilon ]\right) \cap B_{2\delta} (u)$.
        \vspace{2mm}
        \item \label{DL:ii}Denoting $I^a:=\{ u \in H^1_{rad}(\mathbb R^3): I(u)\leq a\}$ for $a\in \mathbb R$, then $\eta(1, I^{c+\varepsilon}\cap B_{\delta}(u_c) )\subset I^{c-\varepsilon}$.
        \vspace{2mm}
        \item\label{DL:iii}For every $v\in H^1_{rad}(\mathbb R^3)$, $I(\eta(\cdot, v))$ is non increasing.
    \end{enumerate}
    In particular, from \eqref{DL:i} and \eqref{DL:iii} it holds
    \begin{equation}\label{iandiii}
        I(\eta(1,v))\leq I(\eta(0,v)) =I(v),\ \text{for every } v \in H^1_{rad}(\mathbb R^3).
    \end{equation}
    For $(\alpha, \beta) \in  Q$ we have two possibilities; if $\alpha \neq 1$ or $\beta \neq 1$, from \eqref{levelc} and \eqref{iandiii}, 
    \[
        I(\eta(1,h(\alpha,\beta)))\leq I(h(\alpha, \beta))<c.
    \]
    If $(\alpha,\beta)=(1,1)$, so that $h(1,1)=u$, it holds $h(1,1) \in I^{c+\varepsilon} \cap B_\delta (u)$, and by \eqref{DL:ii}
    \[
        I(\eta(1,h(1,1))) \leq c-\varepsilon < c.
    \]
    Then
    \begin{equation}\label{underc2}
        \max _{(\alpha,\beta)\in Q} I(\eta(1,h(\alpha,\beta))) <c.
    \end{equation}
    
    \medskip
    {\bf Claim: }we have
    \begin{equation*}
        \eta(1,h(Q))\cap\mathcal M \neq \emptyset.
    \end{equation*}
    Indeed, for $(\alpha, \beta) \in Q$, we define the functions given by $$\varphi(\alpha,\beta):=\eta(1, h(\alpha, \beta))\in H^1_{rad}(\mathbb R^3),$$ and
    \[
        \Psi(\alpha, \beta):=(\psi_1(\alpha, \beta),\psi_2(\alpha, \beta)):=\left( I'(\varphi (\alpha,\beta))[\varphi (\alpha,\beta) ^+], I'(\varphi  (\alpha,\beta))[\varphi (\alpha,\beta) ^-] \right).
    \]
    The claim holds if there exists $(\alpha _0,\beta _0)\in Q$ such that $\Psi(\alpha _0,\beta _0)=(0,0)$. Since
    \begin{align}\label{alpha}
        \norm{u-h(\alpha,\beta)} &= \norm{(u^+ + u^-)-(\alpha u^+ +\beta u^-)} \nonumber \\ 
        &= \abs{1-\alpha}\norm{u^+}+\abs{1-\beta}\norm{u^-} \nonumber \\
        & \geq \abs{1-\alpha}\norm{u^+} \nonumber \\
        & \geq \abs{1-\alpha} L \nonumber \\
        &> \abs{1-\alpha} 6\delta \nonumber \\
        &> 2\delta \iff \alpha < \frac{2}{3} \text{ or } \alpha > \frac{4}{3},
    \end{align}
    using \eqref{DL:i} and \eqref{alpha}, for $\alpha=\frac{1}{2}$ and for every $\beta \in [\frac{1}{2},\frac{3}{2}]$ we have $\varphi (\frac{1}{2},\beta)=h(\frac{1}{2},\beta)$, so that
    \begin{align*}
        \Psi(\textstyle{\frac{1}{2}}, \beta)&=\left( I'(h(\textstyle{\frac{1}{2}},\beta))[h  (\textstyle{\frac{1}{2}},\beta)^+], I'(h (\textstyle{\frac{1}{2}},\beta))[h  (\textstyle{\frac{1}{2}},\beta)^-] \right)\\
        &=\left( I'(\textstyle{\frac{1}{2}}u^++\beta u^-)[\textstyle{\frac{1}{2}}u^+],I'(\textstyle{\frac{1}{2}}u^++\beta u^-)[\beta u^-]\right).
    \end{align*}

    By \eqref{eq:decomposI'}, \eqref{eq:decomp+} and \eqref{M2} of Lemma \ref{domination1} we infer
    \begin{align*}
        I'(\textstyle{\frac{1}{2}}u^++\beta u^-)[\textstyle{\frac{1}{2}}u^+]&=I'(\textstyle{\frac{1}{2}}u^+)[\textstyle{\frac{1}{2}}u^+]+\frac{\beta^2}{4}\lambda\displaystyle{\int_{\mathbb R^3}\phi_{u^+}(u^-)^2dx}\\
        &\geq I'(\textstyle{\frac{1}{2}}u^+)[\textstyle{\frac{1}{2}}u^+]+\textstyle{\left(\frac{1}{2}\right)^4}\lambda\displaystyle{\int_{\mathbb R^3}\phi_{u^+}(u^-)^2dx}\\
        &=\gamma_+(\textstyle{\frac{1}{2}}u)>0,
    \end{align*}
    from which we obtain
    \begin{equation}\label{squareleft}
        \psi_1(\textstyle{\frac{1}{2}}, \beta)= I'(\textstyle{\frac{1}{2}}u^++\beta u^-)[\textstyle{\frac{1}{2}}u^+]>0, \quad \text{for every } \beta \in [\textstyle{\frac{1}{2}},\textstyle{\frac{3}{2}}].
    \end{equation}
    If $\alpha=\frac{3}{2}$, for every $\beta \in [\frac{1}{2},\frac{3}{2}]$, using \eqref{DL:i} above and \eqref{alpha}, we have $\varphi (\frac{3}{2},\beta)=h(\frac{3}{2},\beta)$, and 
    arguing as before we get
    \begin{align*}
        I'(\textstyle{\frac{3}{2}}u^++\beta u^-)[\textstyle{\frac{3}{2}}u^+]&=I'(\textstyle{\frac{3}{2}}u^+)[\textstyle{\frac{3}{2}}u^+]+\textstyle{\left(\frac{3}{2}\right)}^2\beta^2\lambda\displaystyle{\int_{\mathbb R^3}\phi_{u^-}(u^+)^2dx}\\
        &\leq I'(\textstyle{\frac{3}{2}}u^+)[\textstyle{\frac{3}{2}}u^+]+\textstyle{\left(\frac{3}{2}\right)^4}\lambda\displaystyle{\int_{\mathbb R^3}\phi_{u^-}(u^+)^2dx}\\
        &=\gamma_+(\textstyle{\frac{3}{2}}u)<0,
    \end{align*}
    so that
    \begin{equation}
        \psi_1(\textstyle{\frac{3}{2}}, \beta)= I'(\textstyle{\frac{3}{2}}u^++\beta u^-)[\textstyle{\frac{3}{2}}u^+] <0, \quad \text{for every } \beta \in [\textstyle{\frac{1}{2}},\textstyle{\frac{3}{2}}].
    \end{equation}
    
    Analogously we have, by using \eqref{eq:decomposI'}, \eqref{eq:decomp-} and \eqref{M2} of Lemma \ref{domination1},
    \begin{align}
        \psi_2(\alpha,\textstyle{\frac{1}{2}})&= I'(\textstyle{\alpha u^++\frac{1}{2}}u^-)[\textstyle{\frac{1}{2}}u^-] >0, \quad \text{for every } \alpha \in [\textstyle{\frac{1}{2}},\textstyle{\frac{3}{2}}],\\ 
        \vspace{3mm}
        \psi_2(\alpha,\textstyle{\frac{3}{2}})&= I'(\textstyle{\alpha u^++\frac{3}{2}}u^-)[\textstyle{\frac{3}{2}}u^-] <0, \quad \text{for every } \alpha \in [\textstyle{\frac{1}{2}},\textstyle{\frac{3}{2}}]. \label{up}
    \end{align}
    Since $\Psi$ is continuous on $Q$, because $\eta$, $h$ are continuous, and \eqref{squareleft}-\eqref{up} hold, by Miranda's theorem (see \cite{M}), there is $(\alpha _0,\beta _0)\in Q$ such that $\Psi(\alpha _0,\beta _0)=(0,0)$. Therefore, the Claim
      holds.
    
    But then $I(\varphi(\alpha _0,\beta _0))\geq \min_{\mathcal M} I=c$, in contradiction with \eqref{underc2}. Therefore $I'(u)=0$ and the proof is completed.
\end{proof}

\section{Proof of the main result} \label{sec:proof}
Before proving our result let us observe the following:
    \begin{equation}\label{eq:delta}
        \exists \delta >0 \text{ such that } \forall w\in H^1_{rad}(\mathbb R^3) \text{ with } \norm{w}\leq\delta : \gamma_{\pm}(w)=I'(w)[w^{\pm}]\geq \frac{1}{4} \norm{w^{\pm}}^2.    
    \end{equation}
    In fact, fixing $0< \varepsilon < 1/2$, by \eqref{III.4} we have
    \begin{align*}
        \gamma_{\pm}(w) &\geq I'(w^{\pm})[w^{\pm}]\geq\norm{w^{\pm}}^2 -\int_{\mathbb R^3} f(w^{\pm})w^{\pm} dx\\
         &\geq  \norm{w^{\pm}}^2 -\varepsilon\int_{\mathbb R^3}(w^{\pm})^2 dx-C_\varepsilon\int_{\mathbb R^3}\abs{w^{\pm}}^q dx\\
         &= \frac{1}{2}\norm{w^{\pm}}^2 - C_\varepsilon \int_{\mathbb R^3} \abs{w^{\pm}}^q dx+\frac{1}{2}\int_{\mathbb R^3}\abs{\nabla w^{\pm}}^2 dx +\left(\frac{1}{2}-\varepsilon\right)\int_{\mathbb R^3}(w^{\pm})^2 dx\\
         &\geq \frac{1}{4}\norm{w^{\pm}}^2 + \left( \frac{1}{4}\norm{w^{\pm}}^2 - C_\varepsilon \norm{w^{\pm}}^q \right).
    \end{align*}
Choosing $\delta\leq1/(4C_\varepsilon)^{1/(q-2)}$ we conclude.

\medskip

\subsection{Proof of Theorem \ref{theorem}}

    Denote $c=\inf_{\mathcal M} I>0$ and let $\{u_n\} \subset\mathcal M$ be a minimizing sequence, i.e., $I(u_n)\rightarrow c$. By Proposition \ref{boundedsequence} $\{u_n\}$ is bounded and we can  assume that $u_n \rightharpoonup u$ in $H^1_{rad}(\mathbb R^3)$. Since $u_n\in\mathcal M$, $\gamma_+(u_n)=I'(u_n)[u_n^+]=0$ and $\gamma_-(u_n)=I'(u_n)[u_n^-]=0$. Then using \eqref{eq:decomp+} we get
    $$0=\gamma_{+}(u_{n})= \gamma(u_{n}^{+}) + \lambda \int_{\mathbb R^{3}} \phi_{u_{n}^{-}}(u_{n}^{+})^{2}dx$$
    implying that $\gamma(u_{n}^{+})<0$. Similarly, by using \eqref{eq:decomp-}, we get $\gamma(u_{n}^{-})<0$; but then
    by Remark \ref{rem:less}, $\abs{u_n^+}_q,\abs{u_n^-}_q >L>0$, and using the compact embedding $H^1_{rad}(\mathbb R^3) \hookrightarrow \hookrightarrow L^q(\mathbb R^3)$, we have $\abs{u^+}_q,\abs{u^-}_q \geq L>0$. Hence $u^+,u^-\not\equiv 0$, and $u=u^+ +u^-$ is a sign-changing function. 
    
   The function $u$ is the candidate to be the element where the minimum of $I$ on $\mathcal M$ is attained; 
   however for this we need to show that $u\in \mathcal M$.
    
    Note that $$u_n^+ \rightharpoonup u^+  \ \text{ and } \ \ u_n^- \rightharpoonup u^- \text{ in } \ H^1_{rad}(\mathbb R^3).$$
     Indeed, $\norm{u_n^\pm}\leq \norm{u_n}$ so that the sequences $\{u_n^+\}$, $\{u_n^-\}$ are also bounded. Assume that $u_n^+ \rightharpoonup v\geq0$ and $u_n^- \rightharpoonup w\leq0$ in $H^1_{rad}(\mathbb R^3)$. Thus $u_n=u_n^+ +u_n^- \rightharpoonup v + w$ in $H^1_{rad}(\mathbb R^3)$, and then $u=v+ w$, from which $u^+=v, u^{-}=w$. 
    
    \medskip
    
  {\bf Claim:} it holds
    \begin{equation}\label{strongconvergence+}
        u_n^\pm \rightarrow u^\pm \text{ in } H^1_{rad}(\mathbb R^3).
    \end{equation}
    
    We just prove the claim concerning the positive parts, since similarly it can be shown that $u_n^- \rightarrow u^-$ in $H^1_{rad}(\mathbb R^3)$.
    
    Suppose that this is not true, i.e., $\norm{u^+}< \liminf_{n\rightarrow \infty} \norm{u_n^+}$. Since $u_n^\pm \rightharpoonup u^\pm$ in $H^1_{rad}(\mathbb R^3)$, by Remark \ref{convergences}
    \[
        \int_{\mathbb R^3} f(u_n^\pm) u_n^\pm dx \rightarrow  \int_{\mathbb R^3} f(u^\pm) u^\pm dx,\qquad\int_{\mathbb R^3} F(u_n^\pm) dx \rightarrow \int_{\mathbb R^3} F(u^\pm) dx,
    \]
    besides of
    \[
        \int_{\mathbb R^3} \phi_{u_n^+} (u_n^\pm)^2 dx \rightarrow \int_{\mathbb R^3} \phi_{u^+} (u^\pm)^2 dx\quad\text{and}\quad\int_{\mathbb R^3} \phi_{u_n^-} (u_n^\pm)^2 dx \rightarrow \int_{\mathbb R^3} \phi_{u^-} (u^\pm)^2 dx.
    \]
    Therefore
    \begin{align}\label{negativegamma+}
        \gamma_+(u)< \liminf_{n\rightarrow \infty} \left[ \norm{u_n^+}^2+ \lambda \int_{\mathbb R^3} \phi_{u_n^+}u_n^2 dx -\int_{\mathbb R^3} f(u_n^+) u_n^+ dx\right]
        = \liminf_{n\rightarrow \infty} \gamma_+(u_n)=0. 
    \end{align}
    Also recall that $\gamma_-(u_n)=0$, and since
    $$\norm{u^-}^2\leq\liminf_{n\rightarrow\infty}\norm{u_n^-}^2,$$
    we have
    \begin{align}\label{negativegamma-}
        \gamma_-(u)\leq\liminf_{n\rightarrow\infty}\left[\norm{u_n^-}^2+\lambda\int_{\mathbb R^3}\phi_{u_n^-}u_n^2dx-\int_{\mathbb R^3}f(u_n^-) u_n^-dx\right]
        =\liminf_{n\rightarrow \infty} \gamma_-(u_n)=0.
    \end{align}
    Now we take $\delta>0$ satisfying \eqref{eq:delta} 
    and we choose $\zeta>0$ satisfying $\zeta\norm{u}\leq\delta$. Thence
    $$\gamma_+(\zeta u^++tu^-)\geq\frac{\zeta^2}{4}\norm{u^+}^{2}>0,\quad\text{ for every }t\in[\zeta,1],$$
    and
    $$\gamma_-(su^++\zeta u^-)\geq\frac{\zeta^2}{4}\norm{u^-}^{2}>0,\quad\text{ for every }s\in[\zeta,1].$$
    On the other hand, by \eqref{negativegamma+}, for every $t\in[\zeta,1]$ we have
    \begin{align*}
        \gamma_+(u^++tu^-)&=I'(u^+)[u^+]+t^2\lambda\int_{\mathbb R^3}\phi_{u^+}(u^-)^2dx\\
        &\leq I'(u^+)[u^+]+\lambda\int_{\mathbb R^3}\phi_{u^+}(u^-)^2dx=\gamma_+(u)<0,
    \end{align*}
    and by \eqref{negativegamma-}, for every $s\in[\zeta,1]$
    \begin{align*}
        \gamma_-(su^++u^-)&=I'(u^-)[u^-]+s^2\lambda\int_{\mathbb R^3}\phi_{u^+}(u^-)^2dx\\
        &\leq I'(u^-)[u^-]+\lambda\int_{\mathbb R^3}\phi_{u^+}(u^-)^2dx=\gamma_-(u)\leq0.
    \end{align*}
    Hence, by Miranda's theorem, there is a point $(\alpha,\beta)\in[\zeta,1]\times[\zeta,1]$ for which the function given by
    $$\Phi(s,t):=\left(\gamma_+(su^++tu^-),\gamma_-(su^++tu^-)\right),\quad (s,t)\in[\zeta,1]\times[\zeta,1],$$
    satisfies $\Phi(\alpha,\beta)=(0,0)$, i.e., $\alpha u^++\beta u^-\in\mathcal M$. Therefore using \eqref{f_{2}}
    \begin{eqnarray*}
        I(\alpha u^+ + \beta u^-)&=& I(\alpha u^+ + \beta u^-)-\frac{1}{4}I'(\alpha u^+ + \beta u^-)[\alpha u^+ + \beta u^-]\\
        &=&\frac{\alpha^2}{4}\norm{u^+}^2+\frac{\beta^2}{4}\norm{u^-}^2+\int_{\mathbb R^3}\left[\frac{1}{4}f(\alpha u^+)\alpha u^+-F(\alpha u^+)\right]dx\\
        &&+\int_{\mathbb R^3}\left[\frac{1}{4}f(-\beta u^-)(-\beta u^-)-F(-\beta u^-)\right]dx,
    \end{eqnarray*}
    \color{black}
    and since $0<\alpha<1$, $0<\beta\leq1$ and \eqref{eq:strictincreasing} holds,
     we have (again by \eqref{f_{2}})
    \begin{align*}
        I(\alpha u^+ + \beta u^-)&<\frac{1}{4}\norm{u}^2+\int_{\mathbb R^3}\left[\frac{1}{4}f(u)u-F(u)\right]dx\\
        &<\liminf_{n\rightarrow\infty}\frac{1}{4}\norm{u_n}^2+\liminf_{n\rightarrow\infty}\int_{\mathbb R^3}\left[\frac{1}{4}f(u_n)u_n-F(u_n)\right]dx,
    \end{align*}
    where the last inequality follows from our assumption and Fatou's lemma. Thus
    \begin{align*}
        I(\alpha u^+ + \beta u^-)&< \liminf_{n\rightarrow \infty}\left(I(u_n)-\frac{1}{4}I'(u_n)[u_n]\right)\\
        &= \liminf_{n\rightarrow \infty} I(u_n)\\
        &=c
    \end{align*}
    an absurd. Then \eqref{strongconvergence+} follows. 
    
    \medskip
    
    In virtue of the Claim, $u_n \rightarrow u$ in $H^1_{rad}(\mathbb R^3)$. Since $\mathcal M$ is closed in $H^1_{rad}(\mathbb R^3)$, $u\in\mathcal M$, so that $I(u)=c=\inf_{\mathcal M} I$. By Proposition \ref{attained}, $u$ is a critical point of $I$, and thus, a radial sign-changing solution of \eqref{ourproblem} in $H^1_{rad}(\mathbb R^3)$
    which has minimal energy among all the radial and sign-changing solutions.

    \medskip
    
    Let us see that $u$ changes the  sign exactly once. We adapt the proof from the arguments  given in \cite{ASS} and \cite{CCN}.
    In fact by the regularity of the solution  $u$ (see e.g. \cite{R}), the set $E:=\{x\in \mathbb R^3:u(x)\neq0\}$ is open. If $E$ has more than two components, since $u$ changes of sign, without lost of generality we can assume the existence of connected components $E_1$, $E_2$ and $E_3$ such that
    we have the decomposition $u=u_1+u_2+u_3$ with
    \begin{align*}
        u_1&>0\text{ on }E_1,\text{ and }u_1=0\text{ on }E_2\cup E_3,\\
        u_2&<0\text{ on }E_2,\text{ and }u_2=0\text{ on }E_1\cup E_3,\\
        u_3&\neq0\text{ on }E_3,\text{ and }u_3=0\text{ on }E_1\cup E_2.
    \end{align*}
    Also we define $v:=u_1+u_2$, so that $v^+=u_1$ and $v^-=u_2$. Since $u$ is critical point, $I'(u)=0$, from which (here we use the notation $\phi_3:=\phi_{u_3}$)
    \begin{align*}
        I'(u)[u_1]&=I'(u)[v^+]=I'(v)[v^+]+\lambda\int_{\mathbb R^3}\phi_3(v^+)^2dx=0,\\
        I'(u)[u_2]&=I'(u)[v^-]=I'(v)[v^-]+\lambda\int_{\mathbb R^3}\phi_3(v^-)^2dx=0,\\
        I'(u)[u_3]&=0,
    \end{align*}
    and hence, in particular,
    \begin{equation}\label{negativesides}
    \gamma_+(v)=-\int_{\mathbb R^3}\phi_3u_1^2dx<0\quad\text{and}\quad\gamma_-(v)=-\int_{\mathbb R^3}\phi_3u_2^2dx<0.
    \end{equation}
    Let $\delta>0$ as in \eqref{eq:delta} and let $\zeta>0$ be such that $\zeta\norm{v}\leq\delta$. Then for every $t\in[\zeta,1]$
    $$\gamma_+(\zeta v^++tv^-)\geq\frac{\zeta^2}{4}\norm{v^+}^2>0,$$
    and, by \eqref{negativesides},
    \begin{align*}
        \gamma_+(v^++tv^-)&=I'(v^+)[v^+]+t^2\lambda\int_{\mathbb R^3}\phi_{v^-}(v^+)^2dx\\
        &\leq I'(v^+)[v^+]+\lambda\int_{\mathbb R^3}\phi_{v^-}(v^+)^2dx=\gamma_+(v)<0.
    \end{align*}
    Similarly, for every $s\in[\zeta,1]$ 
    $$\gamma_-(sv^++\zeta v^-)\geq\frac{\zeta^2}{4}\norm{v^-}^2>0,$$
    and using again \eqref{negativesides}
    \begin{align*}
        \gamma_-(sv^++v^-)&=I'(v^-)[v^-]+s^2\lambda\int_{\mathbb R^3}\phi_{v^-}(v^+)^2dx\\
        &\leq I'(v^-)[v^-]+\lambda\int_{\mathbb R^3}\phi_{v^-}(v^+)^2dx=\gamma_-(v)<0.
    \end{align*}
    Summarizing, Miranda's Theorem can be applied to the function $\Psi$ defined on $[\zeta,1]\times[\zeta,1]$ by
    $$\Psi(s,t):=\left(\gamma_+(sv^++tv^-),\gamma_-(sv^++tv^-)\right),$$
    so there is a point $(\alpha,\beta)\in[\zeta,1]\times[\zeta,1]$ such that
    $$\gamma_+(\alpha v^++\beta v^-)=\gamma_-(\alpha v^++\beta v^-)=0,$$
    i.e., $\alpha v^++\beta v^-\in\mathcal M$. On the other hand, the assumptions in Remark \ref{domination2} are satisfied, so that
    $$I(\alpha v^++\beta v^-)=I(\alpha u_1+\beta u_2)<I(u_1+u_2+u_3)=I(u)=c=\inf_{\mathcal M}I,$$
    a contradiction. Therefore $E$ has exactly two connected components, and this completely proves Theorem
    \ref{theorem}.

\section*{Appendix}

In this appendix we prove Proposition  \ref{prop:M}, however some preliminaries are in order.
In all that follows,   $\mathfrak u$ denotes
the solution obtained in Lemma \ref{le:facile}
 in correspondence of a small fixed value of $\lambda.$ 

Let us start by observing that  \eqref{eq:disug} reads as
$$\norm{\mathfrak u}^2=\int_{\mathbb R^3}(f(\mathfrak u)\mathfrak u-\lambda\phi_{\mathfrak u}\mathfrak  u^2)dx<\int_{\mathbb R^3}(\mathfrak u^4-\lambda\phi_{\mathfrak u} \mathfrak u^2)dx,$$
so that
\begin{equation}\label{eq:possible}
0<\int_{\mathbb R^3}(\mathfrak u^2-\lambda\phi_{\mathfrak u})\mathfrak u^2dx.
\end{equation}
\color{black}

In the following we will always adopt the convention that, for a radial function $z:\mathbb R^{3}\to \mathbb R$,
we use the same notation to denote the function $r\in [0,+\infty) \mapsto z(r)\in \mathbb R$ where $r=|x|$.
It will be clear from the context if we mean $z(x)$ or $z(r)$. 
As usual we will denote always by $\phi_{z}$ the solution of
$-\Delta\phi=z^{2}$ in $\mathbb R^{3}.$
In particular the above convention  applies to $\mathfrak u$
and $\phi_{\mathfrak u}$. 

By the regularity 
of $\mathfrak u$ and $\phi_\mathfrak u$, there are $r_1,r_4$ real numbers with $0<r_1<r_4$ such that
\begin{equation*}
    \mathfrak u^2-\lambda\phi_{\mathfrak u}>0\quad\text{ on }[r_1,r_4].
\end{equation*}
\color{black}

Given $0<s<\rho$, we use the notation $A_{s,\rho}$ for the annulus $B_\rho\setminus\overline{B}_s$. 
Here  $B_{a}$ denotes the ball centred in zero with radius $a>0$ and $\overline B_{a}$ its closure.

\medskip

Let $\delta >0$ be such that 
$$\int_{A_{r_1,r_4}}(\mathfrak u^2-\lambda\phi_{\mathfrak u})\mathfrak u^2dx>\frac{3}{2} \delta,$$
which is possible by \eqref{eq:possible}. We can find numbers $r_2,r_3>0$ with $r_1<r_2<r_3<r_4$ such that
\begin{equation}\label{condition}
    \int_{A_{r_2,r_3}}(\mathfrak u^2-\lambda\phi_{\mathfrak u})\mathfrak u^2dx>\delta\quad\text{and}\quad\int_{A_{r_i,r_{i+1}}}(\mathfrak u^2+\lambda\phi_{\mathfrak u})\mathfrak u^2dx<\frac{\delta}{4},\quad\text{for }i=1,3.
\end{equation}
\color{black}
Let $\nu, \eta \in C^{\infty} (\mathbb R^3;[0,1])$ be radial cut-off functions satisfying: 
\begin{itemize}
\item 
$\nu =0$ in $B_{r_1}$, $\nu$ strictly increasing (in the radial coordinate) in $A_{r_1,r_2}$,
 with $\nu=1$ outside $B_{r_2}$. 
\item 
$\eta =1$ in $B_{r_3}$, $\eta$ strictly decreasing (in the radial coordinate) in $A_{r_3,r_4}$, with $\eta=0$ outside $B_{r_4}$.
\end{itemize}
For brevity we define the functions
 \begin{equation}\label{function}
    \mathfrak v:= \nu\mathfrak u\eta \ \text{ and } \  e_t:=t \mathfrak v,\quad \text{for}\quad t>1.
 \end{equation}
 We also define the functional $G:H^1_{rad}(\mathbb R^3) \rightarrow \mathbb R$ by
 $$G(u):=\int_{\mathbb R^3} (\abs{\nabla u}^2+2u^2)dx +\lambda\int_{\mathbb R^3} \phi_uu^2dx- \int_{\mathbb R^3}f(u)udx.$$
 
\begin{lem}\label{lem:T0}
With the above notations, there exists $T_{1}>1$ such that 
$$\forall t\geq T_{1} : \ G(e_{t})<0.$$
\end{lem}
\begin{proof}
Computing we have, for $t>1$,
\begin{align*}
    G(e_t)&= t^2\left(\norm{\mathfrak v}^2+\int_{\mathbb R^3}\mathfrak v^2dx\right)+\lambda t^4\int_{\mathbb R^3}\phi_{\mathfrak v}\mathfrak v^2 dx 
     -\int_{\mathbb R^3}f(t\mathfrak v)t\mathfrak vdx
\end{align*}
so that
$$\frac{G(e_t)}{t^4}<\frac{1}{t^2}\left(\norm{\mathfrak u}^2+\int_{\mathbb R^3}\mathfrak u^2dx\right)+\lambda\int _{\mathbb R^3}\phi_{\mathfrak v}\mathfrak v^2 dx-\int_{\mathbb R^3}\frac{f(t\mathfrak v)}{t^3}\mathfrak vdx,$$
from which
\begin{align*}
    \limsup _{t\rightarrow \infty} \frac{G(e_t)}{t^4} & 
    \leq\lambda\int _{\mathbb R^3}\phi_{\mathfrak v}\mathfrak v^2 dx +\limsup_{t\rightarrow \infty}\left[-\int_{\mathbb R^3}\frac{f(t\mathfrak v)}{t^3}\mathfrak vdx\right]\\
    &=\lambda\int _{\mathbb R^3}\phi_{\mathfrak v}\mathfrak v^2 dx-\liminf_{t\rightarrow \infty} \int_{\mathbb R^3}\frac{f(t\mathfrak v)}{t^3}\mathfrak vdx\\
    &\leq\lambda\int _{\mathbb R^3}\phi_{\mathfrak v}\mathfrak v^2 dx-\int_{A_{r_1,r_4}}\left[\liminf_{t\rightarrow\infty}\frac{f(t\mathfrak v)}{(t\mathfrak v)^3}\right]\mathfrak v^4dx\\
    &=\int_{A_{r_1,r_4}}\left(\lambda\phi_{\mathfrak v}-\mathfrak v^2\right)\mathfrak v^2 dx.
\end{align*}
The last inequality above follows from Fatou's lemma. Therefore
\begin{align*}
    \limsup _{t\rightarrow \infty} \frac{G(e_t)}{t^4}&\leq\int_{A_{r_1,r_2}} \left(\lambda\phi_{\nu\mathfrak u}-\nu^2\mathfrak u^2\right)\nu^2\mathfrak u^2 dx+ \int_{A_{r_2,r_3}}\left(\lambda\phi_{\mathfrak u}-\mathfrak u^2\right)\mathfrak u^2dx\\
    &\ \ +\int_{A_{r_3,r_4}}\left(\lambda\phi_{\mathfrak u\eta}-\mathfrak u^2\eta^2\right)\mathfrak u^2\eta^2 dx\\
    &=: I_1+I_2+I_3,
\end{align*}
Let us estimate every integral. By \eqref{condition} we have
$$
I_{1}\leq \int_{A_{r_1,r_2}}\abs{ \lambda \phi_{\nu \mathfrak u}-\nu ^2 \mathfrak u^2} \nu ^2 \mathfrak u^2 dx
\leq \int_{A_{r_1,r_2}}(\mathfrak u^2+\lambda \phi _{\mathfrak u})\mathfrak u^2dx
 < \frac{\delta}{4}.
$$
Similarly
\[
I_{3}
< \frac{\delta}{4},
\]
and, again by \eqref{condition},
$$I_{2}=\int_{A_{r_{2}, r_{3}}} (\lambda \phi_{\mathfrak u} -\mathfrak u^2)\mathfrak u^2 dx<-\delta.$$
Then
\[
\limsup _{t\rightarrow \infty} \frac{G(e_t)}{t^4} < \frac{\delta}{4} -\delta +\frac{\delta}{4}=-\frac{\delta}{2},
\]
from which the conclusion follows.
\end{proof}
For $t>0$, define the functional $H_t:H^1_{rad}(\mathbb R^3) \rightarrow \mathbb R$ by
$$H_t(u):=t\int_{\mathbb R^3}\abs{\nabla u}^2dx+\left(\frac{1}{t}+t^2\right)\int_{\mathbb R^3}u^2dx+\frac{\lambda}{t}\int_{\mathbb R^3}\phi_uu^2dx-\frac{1}{t^2}\int_{\mathbb R^3}f(tu)udx.$$
\begin{lem}\label{lem:T2}
With the notations in \eqref{function}, there exists $T_{2}>1$ such that
$$\forall t\geq T_{2}:\ H_{t}(e_t)<0.$$
\end{lem}
\begin{proof}
    Note that
    \begin{align*}
        H_t(e_t)&=t\int_{\mathbb R^3}\abs{\nabla e_t}^2dx+\left(\frac{1}{t}+t^2\right)\int_{\mathbb R^3}e_t^2dx+\frac{\lambda}{t}\int_{\mathbb R^3}\phi_{e_t}e_t^2dx-\frac{1}{t^2}\int_{\mathbb R^3}f(te_t)e_tdx\\
        &=t^3\int_{\mathbb R^3}\abs{\nabla \mathfrak v}^2dx+\left(t+t^4\right)\int_{\mathbb R^3}\mathfrak v^2dx+t^3\lambda\int_{\mathbb R^3}\phi_{\mathfrak v}\mathfrak v^2dx-\frac{1}{t}\int_{\mathbb R^3}f(t^2\mathfrak v)\mathfrak vdx,
    \end{align*}
    thus
    $$\frac{H_t(e_t)}{t^5}=\frac{1}{t^2}\int_{\mathbb R^3}\abs{\nabla \mathfrak v}^2dx+\left(\frac{1}{t^4}+\frac{1}{t}\right)\int_{\mathbb R^3}\mathfrak v^2dx+\frac{\lambda}{t^2}\int_{\mathbb R^3}\phi_{\mathfrak v}\mathfrak v^2dx-\int_{\mathbb R^3}\frac{f(t^2\mathfrak v)}{t^6}\mathfrak vdx,$$
    and then
    \begin{align*}
        \limsup _{t\rightarrow \infty}\frac{H_t(e_t)}{t^5}&=\limsup _{t\rightarrow \infty}\left[-\int_{\mathbb R^3}\frac{f(t^2\mathfrak v)}{t^6}\mathfrak vdx\right]\\
        &=-\liminf_{t\rightarrow \infty}\int_{\mathbb R^3}\frac{f(t^2\mathfrak v)}{t^6}\mathfrak vdx\\
        &\leq-\int_{A_{r_1,r_4}}\liminf_{t\rightarrow \infty}\left[\frac{f(t^2\mathfrak v)}{(t^2\mathfrak v)^3}\right]\mathfrak v^4dx\\
        &=-\int_{A_{r_1,r_4}}\mathfrak v^4dx,
    \end{align*}
    where the last inequality follows from Fatou's lemma. Therefore we conclude.
\end{proof}
We can prove now  that $\mathcal M$ is nonempty.
\subsection{Proof of Proposition \ref{prop:M}}
    Under the above notations,
    consider the element 
    $$u:=T_{0}\nu \mathfrak u\eta\in H^1_{rad}(\mathbb R^3),$$ 
    where $T_{0}>1$ 
     is  chosen  such that
    \begin{equation}\label{R}
    T_{0}\geq\max\{T_1,T_2\},\quad\left[\frac{r_1}{T_{0}},\frac{r_4}{T_{0}}\right]\cap[r_1,r_4]=\emptyset,\quad\lambda\phi_u(T_{0}r_1)<1,
    \end{equation}
    with $T_{1}, T_{2}$ given in  Lemmas \ref{lem:T0} and \ref{lem:T2}.
    Note that $\supp(u)=A_{r_1,r_4}$. 
    It will be useful the rescaled function $$w(x):=u(T_{0}x).$$
    Before to proceed with the proof, let us show other  preliminary facts.
    
  First note that,
    $$\int_{\mathbb R^3}\abs{\nabla w}^2dx=\frac{1}{T_{0}}\int_{\mathbb R^3}\abs{\nabla u}^2dx,\quad \int_{\mathbb R^3}w^2dx=\frac{1}{T_{0}^3}\int_{\mathbb R^3}u^2dx,$$
    that the nonlinear terms are
    $$\int_{\mathbb R^3}F(w)dx=\frac{1}{T_{0}^3}\int_{\mathbb R^3}F(u)dx,\quad \int_{\mathbb R^3}f(w)wdx=\frac{1}{T_{0}^3}\int_{\mathbb R^3}f(u)udx,$$
    and that the nonlocal term can be written as
    $$\int_{\mathbb R^3}\phi_ww^2dx=\frac{1}{T_{0}^5}\int_{\mathbb R^3}\phi_uu^2dx\quad\text{ since }\quad\phi_w(x)=\frac{1}{T_{0}^2}\phi_u(T_{0}x).$$
  
    Moreover, recalling that $\phi_{u}$ is decreasing and \eqref{R}, we get
    \begin{eqnarray}\label{Rr}
            \lambda\int_{\mathbb R^3}\phi_wu^2dx&=& 4\pi \lambda\int_0^\infty\phi_w(r)u^2(r)r^2dr=\frac{4\pi}{T_{0}^2}\int_{r_1}^{r_4}\lambda\phi_u(T_{0}r)u^2(r)r^2dr \nonumber \\
            &<&\frac{4\pi}{T_{0}^2}\int_{r_1}^{r_4}u^2(r)r^2dr=\frac{1}{T_{0}^2}\int_{\mathbb R^3}u^2dx.
    \end{eqnarray}
    Defining $v:=u-w$, we have $\supp(u)\cap\supp(w)=\emptyset$ by \eqref{R}, so that $v^+=u$ and $v^-=-w$. 

\medskip

    Note that      for every $\tau\in(0,T_{0}]$,
    \begin{align*}
        \begin{split}
            I'(T_{0}v^++\tau v^-)[T_{0}v^+]=&\ T_{0}^2\int_{\mathbb R^3}\left[\abs{\nabla v^+}^2+(v^+)^2\right]dx+T_{0}^4\lambda\int_{\mathbb R^3}\phi_{v^+}(v^+)^2dx\\&+T_{0}^2\tau^2\lambda\int_{\mathbb R^3}\phi_{v^-}(v^+)^2dx-\int_{\mathbb R^3}f(T_{0}v^+)T_{0}v^+dx
        \end{split}\\
        \begin{split}
            \leq&\ T_{0}^2\int_{\mathbb R^3}\left[\abs{\nabla u}^2+u^2\right]dx+T_{0}^4\lambda\int_{\mathbb R^3}\phi_uu^2dx\\&+T_{0}^4\lambda\int_{\mathbb R^3}\phi_wu^2dx-\int_{\mathbb R^3}f(T_{0}u)T_{0}udx,
        \end{split}
    \end{align*}
    from which
    \[
        I'(T_{0}v^++\tau v^-)[T_{0}v^+]\leq\norm{T_{0}u}^2+\lambda\int_{\mathbb R^3}\phi_{T_{0}u}(T_{0}u)^2dx-\int_{\mathbb R^3}f(T_{0}u)T_{0}udx+T_{0}^4\lambda\int_{\mathbb R^3}\phi_wu^2dx.
    \]
   Then we have by \eqref{Rr} and by recalling the notation in \eqref{function} (concretely that  $e_{T_{0}^{2}} = T_{0}u$) 
    \begin{eqnarray}\label{eq:terza}
        I'(T_{0}v^++\tau v^-)[T_{0}v^+]&\leq& \norm{e_{T_{0}^{2}}}^2+\lambda\int_{\mathbb R^3}\phi_{e_{T_{0}^{2}}}e_{T_{0}^{2}}^2dx-\int_{\mathbb R^3}f(e_{T_{0}^{2}})
        e_{T_{0}^{2}}dx+\int_{\mathbb R^3}e_{T_{0}^{2}}^2dx\nonumber\\
        &=&G(e_{T_{0}^{2}})<0,
    \end{eqnarray}
    \color{black}
    by Lemma \ref{lem:T0}. Similarly, for every $\theta\in(0,T_{0}]$,
        \begin{eqnarray*}
            I'(\theta v^++T_{0}v^-)[T_{0}v^-]&=&T_{0}^2\int_{\mathbb R^3}\left[\abs{\nabla v^-}^2+(v^-)^2\right]dx+T_{0}^4\lambda\int_{\mathbb R^3}\phi_{v^-}(v^-)^2dx\\
            &&+T_{0}^2\theta^2\lambda\int_{\mathbb R^3}\phi_{v^+}(v^-)^2dx-\int_{\mathbb R^3}f(T_{0}v^-)T_{0}v^-dx \\
            &\leq& T_{0}^2\int_{\mathbb R^3}\left[\abs{\nabla w}^2+w^2\right]dx+T_{0}^4\lambda\int_{\mathbb R^3}\phi_ww^2dx\\
            &&+T_{0}^4\lambda\int_{\mathbb R^3}\phi_uw^2dx-\int_{\mathbb R^3}f(T_{0}w)T_{0}wdx\\
            &=& T_{0}\int_{\mathbb R^3}\abs{\nabla u}^2dx+\frac{1}{T_{0}}\int_{\mathbb R^3}u^2dx+\frac{\lambda}{T_{0}}\int_{\mathbb R^3}\phi_uu^2dx\\
            &&+T_{0}^4\lambda\int_{\mathbb R^3}\phi_uw^2dx-\frac{1}{T_{0}^3}\int_{\mathbb R^3}f(T_{0}u)T_{0}udx \\
            &\leq& T_{0}\int_{\mathbb R^3}\abs{\nabla u}^2dx+\frac{1}{T_{0}}\int_{\mathbb R^3}u^2dx+\frac{\lambda}{T_{0}}\int_{\mathbb R^3}\phi_uu^2dx\\
            &&+ T_{0}^2\int_{\mathbb R^3}u^2dx-\frac{1}{T_{0}^{3}}\int_{\mathbb R^3}f(T_{0}u)T_{0}udx\\ 
            & = &H_{T_{0}}(u),
        \end{eqnarray*}
 where the last inequality holds true because of \eqref{Rr}. Then, by Lemma \ref{lem:T2}, for every $\theta\in(0,T_0]$
we get:
    \begin{equation}\label{eq:quarta}
    I'(\theta v^++T_0v^-)[T_0v^-]\leq H_{T_0}(u)=H_{T_0}(e_{T_0})<0.
    \end{equation}
    On the other hand, it is clear that
    $$\int_{\mathbb R^3}f(tv^+)tv^+dx=t^2\int_{\supp (u)}\frac{f(tu)}{tu}u^2dx.$$
    By \eqref{f_{3}} and \eqref{f_{4}} it holds
    $$\lim_{t\rightarrow 0}\frac{f(tu)}{tu}=0 \ \text{ and }  \ \frac{f(tu)}{tu}<(tu)^2\quad\text{ a.e. }x\in\supp(u),$$
    and since
    $$\int_{\mathbb R^3}(tu)^2u^2dx=t^2\int_{\mathbb R^3}u^4dx\rightarrow0\quad\text{when}\quad t\rightarrow0,$$
    we deduce
    $$\int_{\mathbb R^3}\frac{f(tu)}{t}udx = \int_{\mathbb R^{3}} \frac{f(t u )}{t u}u^{2}dx< \int_{\mathbb R^{3}} (tu)^{2} u^{2}dx \rightarrow0\quad\text{when}\quad t\rightarrow0.$$
    We infer that 
    \begin{equation}\label{eq:t0}
    \exists t_{0}\in (0,1) : \norm{u}^2>\int_{\mathbb R^3}\frac{f(t_0u)}{t_0}udx.
    \end{equation}
    Furthermore, for every $\tau\in[t_0,T_0]$,
    \color{black}
    we have
    \begin{eqnarray}\label{eq:prima}
            I'(t_0v^++\tau v^-)[t_0v^+]&=&t_0^2\int_{\mathbb R^3}\left[\abs{\nabla v^+}^2+(v^+)^2\right]dx+t_0^4\lambda\int_{\mathbb     R^3}\phi_{v^+}(v^+)^2dx \nonumber\\
            &&+ t_0^2\tau^2\lambda\int_{\mathbb R^3}\phi_{v^-}(v^+)^2dx-\int_{\mathbb R^3}f(t_0v^+)t_0v^+dx \nonumber\\
            &>&t_0^2\left[\norm{u}^2-\int_{\mathbb R^3}\frac{f(t_0u)}{t_0}udx\right] \nonumber\\
            &>&0.
    \end{eqnarray}
   Finally, for every $\theta\in[t_0,T_0]$,
   \begin{eqnarray}\label{eq:seconda}
            I'(\theta v^++t_0v^-)[t_0v^-]&=&t_0^2\int_{\mathbb R^3}\left[\abs{\nabla v^-}^2+(v^-)^2\right]dx+t_0^4\lambda\int_{\mathbb R^3}\phi_{v^-}(v^-)^2dx \nonumber\\
            && + t_0^2\theta^2\lambda\int_{\mathbb R^3}\phi_{v^+}(v^-)^2dx-\int_{\mathbb R^3}f(t_0v^-)t_0v^-dx \nonumber \\
   &>&t_0^2\int_{\mathbb R^3}\left[\abs{\nabla w}^2+w^2\right]dx-\int_{\mathbb R^3}f(t_0w)t_0wdx \nonumber \\
   &>&\frac{t_0^2}{T_0^3}\left[\norm{u}^2-\int_{\mathbb R^3}\frac{f(t_0u)}{t_0}udx\right] \nonumber \\
   &>&0.
   \end{eqnarray}
    Summarizing, there exists $t_0>0$, the one given in \eqref{eq:t0}, and $T_0>t_{0}$, the one given in \eqref{R}, such that, by \eqref{eq:prima} and \eqref{eq:seconda}
    $$I'(t_0v^++\tau v^-)[t_0v^+],\ I'(\theta v^++t_0v^-)[t_0v^-]>0,\quad \forall\theta,\tau\in[t_0,T_0],$$
    and, by \eqref{eq:terza} and \eqref{eq:quarta},
    $$I'(T_0v^++\tau v^-)[T_0v^+],\ I'(\theta v^++T_0v^-)[T_0v^-]<0,\quad \forall\theta,\tau\in[t_0,T_0].$$
    This means that, if we define $\Phi:[t_0,T_0]\times[t_0,T_0]\rightarrow\mathbb R^2$ by
    $$\Phi(\theta,\tau):=\left(I'(\theta v^++\tau v^-)[\theta v^+],\ I'(\theta v^++\tau v^-)[\tau v^-]\right),$$
    as a consequence of Miranda's Theorem, there is $(\alpha_0,\beta_0)\in [t_0,T_0]\times[t_0,T_0]$ that satisfies $\Phi(\alpha_0,\beta_0)=(0,0)$, i.e.
    $$I'(\alpha_0v^++\beta_0v^-)[\alpha_0v^+]=0,\quad I'(\alpha_0v^++\beta_0v^-)[\beta_0v^-]=0.$$
    Therefore $\alpha_0v^++\beta_0v^-\in\mathcal M$ 
    and we conclude the proof.


\end{document}